\documentclass[12pt]{article}
\usepackage{amsmath,amsfonts,amssymb,amsthm}
\usepackage{mathrsfs}
\usepackage{latexsym}
\usepackage[english]{babel}
\usepackage[usenames,dvipsnames]{xcolor}
\usepackage{comment}
\usepackage{appendix}

\usepackage{graphicx}

\renewenvironment{proof}[1][\proofname]{{\bfseries #1.} }{\qed}

\def\Cov{{\rm Cov\,}}

\newcommand{\field}[1]{\mathbb{#1}}

\newcommand{\R}{\field{R}}

\newcommand{\Z}{\field{Z}}

\newcommand{\bE}{\field{E}}

\newcommand{\Var}{{\rm Var}}
\newcommand{\Corr}{{\rm Corr}}

\newcommand{\cC}{{\cal C}}

\newcommand{\cI}{{\cal I}}

\newcommand{\cS}{{\mathcal S}}

\newcommand{\Tb}{{\mathbb{T}}}
\newcommand{\Ec}{{\mathcal{E}}}
\newcommand{\Lc}{{\mathcal{L}}}
\newcommand{\Zc}{{\mathcal{Z}}}

\newcommand{\Nc}{\mathcal{N}}

\newcommand{\dg}{{\dot\gamma}}

\def\authors#1{{ \begin{center} #1 \vspace{0pt} \end{center} } \smallskip}
\def\institution#1{{\sl \begin{center} #1 \vspace{0pt} \end{center} } }
\def\inst#1{\unskip $^{#1}$}
\def\title#1{{\huge\bf  \begin{center} #1 \vspace{0pt} \end{center}  } \smallskip}

\def\cS{{\field S}}
\def\E{{\mathbb{ E}}}

\def\P{{\mathbb{P}}}

\newtheorem{theorem}{Theorem}[section]
\newtheorem{proposition}[theorem]{Proposition}
\newtheorem{lemma}[theorem]{Lemma}

\newtheorem{corollary}[theorem]{Corollary}
\newtheorem{definition}[theorem]{Definition}

\newtheorem{remark}[theorem]{Remark}
\newtheorem{example}[theorem]{Example}

\begin{document}

\date{Nov 2021}

\title{\sc {Correlation Structure and Resonant Pairs for Arithmetic Random Waves}}
\authors{\large Valentina Cammarota\inst{\star}, Riccardo-W. Maffucci\inst{\diamond},\\ Domenico Marinucci\inst{\ddagger}, Maurizia Rossi\inst{\divideontimes}}
\institution{\inst{\star}Dipartimento di Scienze Statistiche, Università di Roma La Sapienza\\
\inst{\diamond}Department of Mathematics, University of Coventry\\
\inst{\ddagger}Dipartimento di Matematica, Università di Roma Tor Vergata\\
\inst{\divideontimes}Dipartimento di Matematica e Applicazioni, Università di Milano-Bicocca
}

\begin{abstract}

The geometry of Arithmetic Random Waves has been extensively investigated in the last fifteen years, starting from the seminal papers \cite{rudwi2,orruwi}. In this paper we study the correlation structure among different functionals such as nodal length, boundary length of excursion sets, and the number of intersection of nodal sets with deterministic curves in different classes; the amount of correlation depends in a subtle fashion from the values of the thresholds considered and the symmetry properties of the deterministic curves. In particular, we prove the existence of \emph{resonant pairs} of threshold values where the asymptotic correlation is full, that is, at such values one functional can be perfectly predicted from the other in the high energy limit. We focus mainly on the $2$-dimensional case but we discuss some specific extensions to dimension $3$.

\smallskip


\smallskip


\end{abstract}

\section{Introduction}

The geometry of nodal sets for Gaussian random eigenfunctions has been the object of a considerable amount of attention over the last $15$ years. Most papers have focussed on the $2$-dimensional case, either in Euclidean settings (Berry's random waves, see  e.g.  \cite{berry2, BCW19,  NourdinPeccatiRossi2019, Vid21}), or on compact manifolds, most notably the sphere $\mathbb{S}^2$ (see e.g. \cite{wispha, MRossiWigman2020, Tod20}) and the torus $\mathbb{T}^2$ (see  e.g. \cite{rudwi2,orruwi, KKW, MPRW, TLR23}), among others. The derivation of the expected value for the nodal length is now  standard thanks to the Gaussian Kinematic Formula by Adler and Taylor (see \cite{adltay}); the analysis of the variance is more challenging, and goes back to \cite{wispha} for the case of the sphere (random spherical harmonics) and to \cite{KKW} for the torus (arithmetic random waves); in the physics literature, the variance for the nodal length of planar eigenfunctions (Berry's random waves) was earlier given in \cite{berry2}. 

The analysis of the asymptotic distribution for the nodal length of random eigenfunctions was basically started in \cite{MPRW}. In that paper, the nodal length of arithmetic random waves is expanded into orthogonal terms corresponding to so-called Wiener chaos components; it is then shown that the behaviour of nodal length is dominated (in the $L^2$ sense, as the eigenvalues diverge) by the fourth-order chaos, whose limiting distribution is nonGaussian in the toroidal case. A similar phenomenon takes place in the planar and spherical cases, see \cite{NourdinPeccatiRossi2019} and \cite{MRossiWigman2020}, respectively, although in both these cases the limiting behaviour is Gaussian.

The expansion into Wiener chaoses has allowed to provide an interpretation to the so-called Berry cancellation phenomenon: namely, the fact that the asymptotic variance of nodal length is an order of magnitude smaller than the variance for the measure of boundary curves at any non-zero thresholds. As argued extensively elsewhere, see \cite{MPRW, MRossiWigman2020, NourdinPeccatiRossi2019, CM18}, this cancellation corresponds to the disappearance of the second-order chaos term, which in turn corresponds to the random $L^2$ norms of eigenfunctions and plays no role in the fluctuations of the nodal lines. Because this random norm has variance which is larger than all the other terms in the orthogonal decomposition, its disappearence in the nodal case fully explains the Berry's cancellation phenomenon.

The domination of this second order chaotic term yields another remarkable consequence, which was already derived, by a different argument, in \cite{Wigman2012}. In particular, the correlation between boundary lengths is asymptotically equal to one (in absolute value) as the eigenvalues diverge. This follows easily by the fact that the random sequences corresponding to boundary lengths at different level are all asymptotically proportional (up to different scaling constants) to the same sequence of random variables, namely the random $L^2$ norms of the eigenfunctions. Asymptotic correlation has been noted in the case of random spherical harmonics \cite{Wigman2012, MarinucciRossi21}, but it obviously holds with exactly the same argument both for planar random waves and for toroidal eigenfunctions. For the same argument, it is also immediate to notice that nodal lengths and the boundary of level curves have asymptotically correlation zero
; indeed nodal length is dominated by terms in the fourth-order chaos, which are by construction orthogonal to the random norm, which belongs to the second-order chaos.

For the reasons we mentioned above, it is clear that these correlation$\slash$uncorrelation phenomena are in some sense an artifact due to the random fluctuations in the $L^2$ norms; for many applications this could sound meaningless (for instance, in a quantum mechanics framework random norms should be normalized to unity). In \cite{MarinucciRossi21}, a different issue was addressed for random spherical harmonics, namely the existence of correlation \emph{after the effect of the random $L^2$- norm has been removed} (usually called partial correlation in mathematical statistics). Surprisingly, it was shown that partial correlation among boundary lengths at different levels persists, and indeed the asymptotic correlation with nodal lengths switches from zero to unity: namely, it is asymptotically possible to predict the boundary length of level curves at every threshold $u$, once the confounding effect of the random norm has been removed.  

The purpose of this paper is to investigate an analogous question, in the case of arithmetic random waves. We focus mainly on the $2$-dimensional case but we discuss some specific
extensions to dimension $3$. The final outcome turns out to be rather different than spherical harmonics and planar random waves. In particular, for dimension $2$, it turns out that, after the effect of the $L^2$ norm has been removed, full correlation between boundary lengths at any different levels does no longer hold: however, there are \emph{resonant pairs} along an algebraic curve (whose equation we write down explicitly), such that the boundary lengths computed at any two levels corresponding to these pairs have  partial correlation tending to $1$ in the high-energy limit. These pairs include the nodal case, where one of the two levels corresponds to $u=0$; moreover, resonant pairs exist (for different algebraic curves, which we also write down explicitly) for other functionals, such as nodal intersections with a fixed reference curve, see \cite{rudwig, Maf18, Cam19, RWY16}.

In order to formulate our results more precisely, we now need to introduce more notation and definitions, which have all become standard in the last decade.

\section{Background and notation}

\subsection{Arithmetic random waves}

We start by recalling the (by now standard) definition of Arithmetic Random Waves, first introduced in  \cite{rudwi2,orruwi}. Let $\Tb^d:=\R^d/\Z^d$ be the standard
$d$-dimensional flat torus and $\Delta$ the Laplacian on $\Tb^d$. We are interested
in the (totally discrete) spectrum of $\Delta$ i.e., eigenvalues $E>0$
of the Helmholtz equation
\begin{equation}
\label{eq:Schrodinger}
\Delta f + Ef=0.
\end{equation}
Let $$S:=\{{{} n \in \Z : n} =  \lambda_1^2+ \cdots + \lambda_d^2 \,\,  \mbox{{} for some} \:\lambda_1, \dots  \lambda_d\in\Z\},$$ be the collection of all numbers
expressible as a sum of two squares. Then, the eigenvalues of \eqref{eq:Schrodinger}
(also called {\it energy levels} of the torus) are all numbers of the form $E_{n}=4\pi^{2}n$, with $n\in S$.

In order to describe the Laplace eigenspace corresponding to $E_{n}$, denote by
$\Lambda_n$ the set of {\it frequencies}:
\begin{equation*}
\Lambda_n := \lbrace \lambda \in \Z^d : || \lambda || ^2 = n\rbrace\
\end{equation*}
whose cardinality
\begin{equation}
\label{eq:Nn=|Lambda|}
\mathcal N_n := |\Lambda_n|
\end{equation}
equals the number of ways to express $n$ as a sum of $d$ squares.
(Geometrically, $\Lambda_{n}$ is the collection of all standard lattice points
lying on the centred circle with radius $\sqrt{n}$.)
For $\lambda\in \Lambda_{n}$ denote the complex exponential associated to the frequency $\lambda$
\begin{equation*}
e_{\lambda}(x) = \exp(2\pi i \langle \lambda, x \rangle)
\end{equation*}
with $x \in\Tb^d$.
Of course, the collection $\{e_{\lambda}(x)\}_{\lambda\in \Lambda_n}$ of the complex exponentials corresponding to the frequencies $\lambda\in\Lambda_{n}$, is an $L^{2}$-orthonormal basis of the eigenspace $\Ec_n$ of $\Delta$ corresponding to the eigenvalue $E_{n}$. In view of \eqref{eq:Nn=|Lambda|}, we have $\dim \Ec_{n}=\mathcal N_n = |\Lambda_n|$; the fluctuations in the number $\Nc_{n}$ have been very widely studied starting from \cite{Lan08}, see for instance \cite{KKW} and the references therein.

Following \cite{rudwi2,orruwi, KKW}, we define
the {\it Arithmetic Random Waves} (also called {\it random Gaussian toral Laplace eigenfunctions})
to be the random fields
\begin{equation}\label{defrf}
T_n(x)=\frac{1}{\sqrt{\mathcal N_n}}\sum_{ \lambda\in \Lambda_n}a_{\lambda}e_\lambda(x), \quad x\in \Tb^d,
\end{equation}
where the coefficients $a_{\lambda}$ are standard complex-Gaussian random variables\footnote{From now on, we assume that every random object considered in this paper is defined on a common probability space $(\Omega, \mathcal{F}, \P)$, with $\E$ denoting mathematical expectation with respect to $\P$.} verifying the following properties: $a_\lambda$ is stochastically independent of $a_\gamma$ whenever $\gamma \notin \{\lambda, -\lambda\}$, and
$a_{-\lambda}= \overline{a}_{\lambda}$ (ensuring that $T_{n}$ is real-valued).
By the definition \eqref{defrf}, $T_n$ is a stationary, i.e. the law of $T_{n}$ is invariant under all the translations $f(\cdot)\mapsto f(x+\cdot)$,
$x \in \Tb^d$, centered
Gaussian random field with covariance function
\begin{equation*}
r_n(x,y) = r_{n}(x-y) := \E[T_n(x) \overline{T_n(y)}] = \frac{1}{\mathcal N_n}
\sum_{\lambda\in \Lambda_n}e_{\lambda}(x-y)=\frac{1}{\mathcal N_n}\sum_{\lambda\in \Lambda_n}\cos\left(2\pi\langle x-y,\lambda \rangle\right),
\end{equation*}
$x,y\in\Tb^d$ (by the standard abuse of notation for stationary fields). Note that $r_{n}(0)=1$, i.e. $T_{n}$ has unit variance. 
\\ 

 The set $\Lambda_n$ induces a discrete probability measure $\mu_n$ on the unit sphere $\mathbb{S}^{d-1}$ 
 $$\mu_n= \frac{1}{ \mathcal N_n} \sum_{\lambda \in \Lambda_n} \delta_{\frac{\lambda}{\sqrt n}}.$$ 
 It turns out that the behavior of  $\lbrace \mu_n\rbrace_n$ strongly depends on the dimension. Indeed,  if $d = 2$, ${\mathcal N_n}$ is subject to large and erratic fluctuations, it grows on average, over integers which are sums of two squares, as ${\rm const} \cdot \sqrt{\log n}$, but can be as small as 8 for an infinite sequence of prime numbers. These erratic fluctuations are mirrored by the behavior of $\lbrace \mu_n\rbrace_n$: indeed,  
 let us denote by
 $$\widehat{\mu}_n(k)=\int_{\cS^1} z^{-k} d \mu_n(z), z \in \mathbb{Z},$$ 
 the Fourier coefficients of $\mu_n$. (We note that 
 $\widehat{\mu}_n(4) \in \mathbb{R}$ since $\mu_n$ is invariant under the transformations $z \to \bar{z}$ and $z \to i \cdot z$, and that 
$|\widehat{\mu}_n(4)| \le 1$ since $\mu_n$ is a probability measure.) Remarkably,
\cite{KKW, kurwig} showed that the set of adherent points of $\{\widehat{\mu}_n(4) \}_{n \in S}$ is all of $[-1,1]$. It is known that for a density-$1$ sequence of eigenvalues, the sequence $\lbrace \mu_n\rbrace_n$ converge towards the uniform probability measure on the circle; weak-* limits of the sequence $\lbrace \mu_n\rbrace_n$ are partially classified in \cite{kurwig}. 

In dimension $d=3$ instead, we have
\[n^{\frac12 - o(1)}\ll \mathcal N_n \ll n^{\frac12 + o(1)},\] and the lattice points $\Lambda_n/\sqrt n$ become equidistributed with respect to the normalized Lebesgue measure on $\mathbb S^2$, as $n\to +\infty$ s.t. $n\not\equiv 0,4,7 (\mathrm{mod} 8)$ \cite{iwniec}.

\subsection{Geometry of level sets}\label{secGeometry}

We are interested in geometric properties of Arithmetic Random Waves, in particular we study the distribution of their level sets 
\begin{equation*}
    \lbrace x\in \mathbb T^d : T_n(x)=u\rbrace,\quad u\in \mathbb R. 
\end{equation*}
(Recall that ARW are a.s. Morse functions, hence the level sets are a.s. smooth submanifolds of codimension $1$.) 
We mainly focus on dimension $d=2$ and investigate two types of functionals: (i) the $1$-dimensional Hausdorff measure, and (ii) the number of  intersection points between nodal sets and a fixed reference curve. More precisely, we shall focus on the following functionals:
\begin{itemize}
\item boundary length at level $u\neq 0$ 
$$\mathcal{L}_{n}^{u}=\mathcal{H}^1\lbrace x\in \mathbb T^2 : T_n(x)=u \rbrace, $$     
\item nodal length 
$$\mathcal{L}_{n}=\mathcal{H}^1\lbrace x\in \mathbb T^2 : T_n(x)=0 \rbrace =: \mathcal{L}_{n}^{0},$$
\item number of intersections of the nodal lines with smooth curves $\mathcal C\subset \mathbb T^2$ with no-where zero curvature:
$$\mathcal{Z}_{n}\mathcal{(C)}=\mathcal H^0\lbrace x\in \mathbb T^2 : T_n(x)=0, \ x \in \mathcal{C} \rbrace.$$

\end{itemize}

In \cite{KKW,  MPRW},  the variance of the boundary length have been investigated; in particular, for $u\ne 0$, as $n\to +\infty$ s.t. $\mathcal N_n\to +\infty$,
\begin{equation}\label{varLu}
    \Var(\mathcal L_n^u) \sim \frac{1}{32} u^4 e^{-u^2}  \frac{E_n}{\mathcal N_n},
\end{equation}
in the nodal case, the variance is of smaller order, indeed 
\begin{equation}\label{varLu}
    \Var(\mathcal L_n) \sim \frac{1+\widehat \mu_n(4)^2}{512} \frac{E_n}{\mathcal N_n^2}.
\end{equation}
As for nodal intersections, we have the following \cite{rudwig}: as $n\to +\infty$ s.t. $\mathcal N_n\to +\infty$,
\begin{equation}\label{varZC}
    \Var(\mathcal Z_n(\mathcal C)) = (4\mathcal B_{\mathcal C}(\mu_n) - L^2) \frac{n}{\mathcal N_n} + O \left ( \frac{n}{\mathcal N_n^{\frac32}}\right ), 
\end{equation}
where $L$ is the length of the curve, and for a probability measure $\nu$ on the unit circle $\cS^1$
\[\mathcal{B}_{\cC}(\nu):=\int_{\cS^1}\left(\int_{0}^{L}\langle\theta,\dg(t)\rangle^2dt\right)^2d\nu(\theta),\]
 $\gamma = (\gamma_1, \gamma_2) :[0,L]\to\Tb^2$ being
the arc-length parametrization of 
$\cC$.
There are special curves for which the leading term in the variance \eqref{varZC} vanishes, as noted in \cite{rudwig}.
\begin{definition}[{\cite[Definition 1]{roswig}}]
A smooth curve $\cC \subset \mathbb T^2$ with nowhere zero curvature and total length $L$ is \emph{static} if
\[\mathcal{B}_{\cC}(\nu):=\int_{\cS^1}\left(\int_{0}^{L}\langle\theta,\dg(t)\rangle^2dt\right)^2d\nu(\theta)=\frac{L^2}{4},\]
for every probability measure $\nu$ on the unit circle $\cS^1$. 
\end{definition}
From \cite[Proposition 7.1]{rudwig}, a curve is static if and only if $\mathcal B_{\mathcal C}\big (  \frac{d\theta}{2\pi}  \big )=0$, where $\frac{d\theta}{2\pi}$ denotes the uniform probability measure on $\cS^1$.

The variance of nodal intersections with static curves has been investigated in \cite{roswig}. We first need some more notation: let $\delta>0$, a sequence of energy levels is called $\delta$-separated \cite{bourud11} if 
$$
\min_{\lambda\ne \lambda'\in \Lambda_n} \| \lambda - \lambda'\| \gg n^{1/4 +\delta}.
$$
For a static curve (that we denote $\mathcal C'$ to avoid confusion), for $\delta$-separated sequence of energy levels such that $\mathcal N_n\to +\infty$   
\begin{equation}
\Var(\mathcal Z_n(\mathcal C')) \sim \frac{n}{4\mathcal N_n^2}(16\mathcal A_{\mathcal C'}(\mu_n) - L^2),
\end{equation}
where for a smooth curve $\mathcal C\subset \mathbb T^2$ and a probability measure $\nu$ on the unit circle $\cS^1$
\[\mathcal{A}_{\cC}(\nu):=\int_{\cS^1}\int_{\cS^1}\left(\int_{0}^{L}\langle\theta,\dg(t)\rangle^2 \langle\theta',\dg(t)\rangle^2dt\right)^2d\nu(\theta)\nu(\theta').\]

Regarding higher dimensions, see Section \ref{3D-results} and Section \ref{3D}. 

\subsection{Chaos expansion}\label{sec_chaos}

The celebrated Wiener chaos expansion  concerns the representation of square
integrable random variables in terms of an infinite orthogonal sum. In this section we recall briefly some basic facts on Wiener chaotic expansion for non-linear functionals of Gaussian fields. Denote by $\{H_k\}_{k \ge 0}$ the Hermite polynomials on $\mathbb{R}$, defined as follows 
\begin{align} \label{hermite}
H_0 =1, \hspace{0.5cm} H_k(t)=(-1)^k \gamma^{-1}(t) \frac{d^k}{d t^k} \gamma(t), \;\; k \ge 1,
\end{align} 
where $\gamma(t)=e^{-t^2/2}/\sqrt{2 \pi}$ is the standard Gaussian density on the real line; $\mathbb{H}= \{ H_k / \sqrt{k!}: \; k \ge 0\}$ is a complete orthogonal system in 
$$L^2(\gamma)= L^2(\mathbb{R}, {\cal B} (\mathbb{R}), \gamma(t) d t).$$

The random eigenfunctions defined in \eqref{defrf} are a byproduct of the family of complex-valued, Gaussian random variables $\{a_{\lambda}\}$, defined on some probability space $(\Omega, {\cal F},\mathbb{P})$.  Define the space ${\bf A}$ to be the closure in $L^2(\mathbb{P})$ generated by all real, finite, linear combinations of random variables of the form $z a_{\lambda}+ \overline z a_{-\lambda}$, $z \in \mathbb{C}$; the space ${\bf A}$ is a real, centred, Gaussian Hilbert subspace of $L^2(\mathbb{P})$.  

For each integer $q \ge 0$, the $q$-th {\it Wiener chaos} ${\cal H}_q$ associated with ${\bf A}$ is the closed linear subspace of $L^2(\mathbb{P})$ generated by all real, finite, linear combinations of random variables of the form 
$$H_{q_1}(a_1) \cdot H_{q_2}(a_2) \cdots H_{q_k}(a_k)$$
for $k \ge 1$, where the integers $q_1, q_2, \dots, q_k \ge 0$ satisfy $q_1+q_2+ \cdots+q_k=q$ and $(a_1,a_2,\dots,a_k)$ is a real, standard, Gaussian vector extracted from ${\bf A}$. In particular ${\cal H}_0=\mathbb{R}$.

As well-known Wiener chaoses $\{{\cal H}_q, \, q=0,1,2, \dots\}$ are orthogonal, i.e., ${\cal H}_q \perp {\cal H}_p$ for $p \ne q$ (the orthogonality holds in the sense of $L^2(\mathbb{P})$) and the following decomposition holds: every real-valued function $F \in {\bf A}$ admits a unique expansion of the type 
\begin{align*} 
F = \sum_{q=0}^{\infty} F[q],
\end{align*}
where the projections $F[q] \in {\cal H}_q$ for every $q=0,1,2,\dots$ and the series converges in $L^2(\mathbb{P})$. Note that $F[0]=\mathbb{E}[F]$.\\

For $u \ne 0$, \cite[Theorem 2.4]{CMR23} proves that the boundary length $\mathcal L_n^u$ is dominated by the second order chaos: 
\begin{align*}
\mathcal L_n^u[2]= \sqrt{\frac{\pi}{8}} u^2 \phi(u) \sqrt{2 \pi^2 n} \frac{1}{{\mathcal N}_n/2} \sum_{\lambda \in \Lambda_n^+} (|a_\lambda|^2-1) + {\cal R}_1(n,u)
\end{align*}
where, under \cite[Condition 2.2]{CMR23}, $\mathbb{E}[{\cal R}_1(n,u)^2 ]=O(4 \pi^2 n/{\mathcal N}_n^2)$, and $\Lambda_n^+$ is the following subset of the set of frequencies: if $n$ is not a square
$$\Lambda_n^+=\{ \lambda=(\lambda_1, \lambda_2) \in \Lambda_n: \lambda_2 >0\},$$ 
otherwise $$\Lambda_n^+=\{ \lambda=(\lambda_1, \lambda_2) \in \Lambda_n: \lambda_2 >0\} \cup \{(\sqrt n, 0)\}.$$ 
Note that for every $n \in S$, $|\Lambda_n^+|= \mathcal N_n/2$. 

Also, inspired by  \cite[Lemma 4.2]{MPRW}, we are able to show that the fourth chaotic component $\mathcal L_n^u[4]$ of the length of $u$-level curves can be written as 
\begin{eqnarray*}
\mathcal L_n^u[4] = \phi(u)\sqrt{\frac{\pi}{2}} \frac{\sqrt{E_n/2}}{\mathcal N_n} \left [ a(u) W_1(n)^2 - \frac14 W_2(n)^2 - \frac14 W_3(n)^2 - \frac12 W_4(n)^2 - \left (a(u) - \frac14 \right ) + o_{\mathbb P}(1) \right],
\end{eqnarray*}
where 
\begin{align*}
&W_1(n)=\frac{1}{n \sqrt{{\mathcal N}_n/2}} \sum_{\lambda \in \Lambda^+_n} (|a_\lambda|^2-1) n,\\
&W_2(n)=\frac{1}{n \sqrt{{\mathcal N}_n/2}}  \sum_{\lambda \in \Lambda^+_n} (|a_\lambda|^2-1) \lambda_1^2, \; \;\; W_3(n)= \frac{1}{n \sqrt{{\mathcal N}_n/2}} \sum_{\lambda \in \Lambda^+_n} (|a_\lambda|^2-1) \lambda_2^2,\\
&W_4(n)= \frac{1}{n \sqrt{{\mathcal N}_n/2}} \sum_{\lambda \in \Lambda^+_n} (|a_\lambda|^2-1) \lambda_1 \lambda_2,
\end{align*}
and
\begin{equation*}
    a(u) = \frac 1 4 H_4(u) + \frac 1 2 H_2(u) - \frac 1 8,
\end{equation*}
and $o_{\mathbb P}(1)$ denotes a sequence of random variables converging to zero in probability. Equivalently, 
\begin{align*}
\mathcal L_n^u[4] = &\phi(u)\sqrt{\frac{\pi}{2}} \frac{\sqrt{E_n/2}}{\mathcal N_n} \Big [ \frac18 \frac{1}{\mathcal N_n/2} \sum_{\lambda,\lambda'\in \Lambda^+_n} (|a_\lambda|^2-1)(|a_{\lambda'}|^2-1) \left ( 8a(u) - 2\left \langle \frac{\lambda}{|\lambda|},\frac{\lambda'}{|\lambda'|}\right \rangle^2  \right ) \\
& \hspace{3cm}  - \left (a(u) - \frac14 \right ) + o_{\mathbb P}(1) \Big ].
\end{align*}
It is known, see \cite[Section 1.4 and Lemma 4.2]{MPRW}, that $\mathcal L_n^0$ is dominated by the fourth chaotic projection:
\begin{eqnarray*}
\mathcal L_n^0[4] = \frac{1}{2} \frac{\sqrt{E_n/2}}{\mathcal N_n} \left [ \frac 1 8 W_1(n)^2 - \frac14 W_2(n)^2 - \frac14 W_3(n)^2 - \frac12 W_4(n)^2 +  \frac 1 8  + o_{\mathbb P}(1) \right].
\end{eqnarray*}
\cite[Section 2.1 and Section 4]{roswig} shows that, in the case
of a non-static curve, the second chaotic projection dominates the chaos expansion of $\mathcal{Z}_n(\mathcal{C})$, and it has the form 
\begin{align*}
\mathcal{Z}_n(\mathcal{C})[2]=\mathcal{Z}^a_n(\mathcal{C})[2] + \mathcal{Z}^b_n(\mathcal{C})[2],  
\end{align*}
where $\Var(\mathcal{Z}^b_n(\mathcal{C})[2])=o(\Var(\mathcal{Z}^a_n(\mathcal{C})[2] ))$, and
\begin{align*}
    \mathcal{Z}^a_n(\mathcal{C})[2]=\frac{\sqrt{2 \pi^2 n}}{2 \pi} \frac{1}{\mathcal N_n/2}  L \sum_{\lambda \in \Lambda^+_n}(|a_\lambda|^2-1) \big( 2  I_{\lambda,\lambda'}(2,0) - 1 \big ),
\end{align*}
where we have introduced the notation
\begin{align*}
   I_{\lambda,\lambda'}(k,k'):=\frac{1}{L}\int_{0}^{L}\left\langle\frac{\lambda}{|\lambda|},\dg(t)\right\rangle^k\left\langle\frac{\lambda'}{|\lambda'|},\dg(t)\right\rangle^{k'}dt. 
\end{align*} 
Now if the curve is static \cite[Lemma 6.5]{roswig},  the leading term in the chaotic expansion of $\mathcal{Z}_n(\mathcal{C}')$ is no longer the projection onto the second chaos, but the projection  onto the fourth chaos: 
\begin{align*}
    \mathcal{Z}_n(\mathcal{C}')[4]= \mathcal{Z}^a_n(\mathcal{C}')[4] + \mathcal{Z}^b_n(\mathcal{C}')[4]
\end{align*}
where $\Var(\mathcal{Z}^b_n(\mathcal{C}')[4])=o(\Var(\mathcal{Z}^a_n(\mathcal{C}')[4] ))$, and 
\begin{align*}
    \mathcal{Z}^a_n(\mathcal{C}')[4] &= \frac{\sqrt{2 n}}{4 \mathcal {N}_{n}} L \Big[\frac{1}{\mathcal {N}_{n}/2} \sum_{\lambda,\lambda'\in\Lambda_{n}^{+}} (|a_\lambda|^2-1)(|a_{\lambda'}|^2-1) (-4 I_{\lambda, \lambda'}(2,2)-1 + 4I_{\lambda, \lambda'}(2,0))  \\
    &\hspace{1.7cm}+
\frac{1}{\mathcal {N}_{n}}\sum_{\lambda\in\Lambda_{n}}(4I_{\lambda, \lambda'}(4,0)-1) \Big]
\end{align*}
but, for any static curve we have that $I_{\lambda, \lambda'}(2,0)=1/2$ via Lemma \ref{lem:sta}, so 
\begin{align*}
    \mathcal{Z}^a_n(\mathcal{C}')[4] &= \frac{\sqrt{2 n}}{4 \mathcal {N}_{n}}L \Big[\frac{1}{\mathcal {N}_{n}/2} \sum_{\lambda,\lambda'\in\Lambda_{n}^{+}} (|a_\lambda|^2-1)(|a_{\lambda'}|^2-1) (1-4 I_{\lambda, \lambda'}(2,2)) +
\frac{1}{\mathcal {N}_{n}}\sum_{\lambda\in\Lambda_{n}}(4I_{\lambda, \lambda'}(4,0)-1) \Big]. 
\end{align*}
And in the particular case of a doubly static curve we have (using \eqref{eq:dou} -- see the proof of Lemma \ref{lem:dou}):
\[I_{\lambda, \lambda'}(4,0)=\frac{3}{8}, \qquad\qquad I_{\lambda, \lambda'}(2,2)= \frac{1}{8} \Big(1+2 \left \langle \frac{\lambda}{|\lambda|}, \frac{\lambda'}{|\lambda'|} \right \rangle^2 \Big),\]
so 
\begin{align*}
    \mathcal{Z}^a_n(\mathcal{C}'')[4] &= \frac{\sqrt{2 n}}{4 \mathcal {N}_{n}} L \Big[\frac{1}{\mathcal {N}_{n}/2} \sum_{\lambda,\lambda'\in\Lambda_{n}^{+}} (|a_\lambda|^2-1)(|a_{\lambda'}|^2-1) \Big( \frac{1}{2} - \Big \langle \frac{\lambda}{|\lambda|}, \frac{\lambda'}{|\lambda'|} \Big \rangle^2  \Big) + \frac 1 2 \Big]. 
\end{align*}

\section{Main results}

\subsection{The correlation structure in the $2$-dimensional case}

Let us start investigating the correlation between boundary lengths.
\begin{theorem}\label{thu1u2}
Let $u_1,u_2\in \mathbb R$, for $n \subset \{S\}$ 
sequence of energies such that  $\mathcal N_n\to +\infty$, 
\begin{equation*}
    \Corr(\mathcal L_n^{u_1}, \mathcal L_n^{u_2}) \to \begin{cases}
        1,\qquad & u_1, u_2\ne 0 \textnormal{ or } u_1=u_2=0, \\
        0,\qquad &\textnormal{otherwise}.
    \end{cases}
\end{equation*}
\end{theorem}
Theorem \ref{thu1u2} follows immediately from \cite[Corollary 2.7]{CMR23} and \cite[Section 1.4]{MPRW}. The correlation structure between boundary length and the intersection number with a fixed smooth reference curve of nowhere zero curvature is given below.  

\begin{theorem} \label{19:07}
Let $u\in \mathbb R$, and $\mathcal{C}\subset \mathbb T^2$ be a smooth curve of total length $L$ with nowhere zero curvature and for which $\lbrace 4\mathcal B_{\mathcal C}(\mu_n)-L^2\rbrace_n$ is bounded away from zero. Then for $n\subset \{S\}$ such that $\mathcal N_n\to +\infty$,
\begin{equation}
    \Corr(\mathcal L_n^u, \mathcal Z_n(\mathcal C)) \to 0.
\end{equation}
\end{theorem}
The proof of Theorem \ref{19:07} is in Section \ref{2D1}. To deal with the static case, we need to define the following
\begin{equation}\label{I'4}
\cI_4'=\cI_{4}'(\mathcal C'):=\frac{1}{L}\int_{0}^{L}(\dg_1(t)^4+\dg_2(t)^4-6\dg_1(t)^2\dg_2(t)^2)dt. 
\end{equation}
\begin{theorem} \label{we201223}
Let $\mathcal{C}'\subset \mathbb T^2$ be a static curve. For $\delta$-separated sequences of energy levels $n \subset \{S\}$ such that $\mathcal N_n\to +\infty$, 
\begin{equation}
    \Corr(\mathcal L_n^u, \mathcal Z_n(\mathcal C')) \to 0,\quad \textnormal{if }u\ne 0.
\end{equation}
If moreover $\widehat \mu_n(4)\to \eta$, we have that 
\begin{equation}
    \Corr(\mathcal L_n, \mathcal Z_n(\mathcal C')) \to f_{\mathcal C'}(\eta),
\end{equation}
where 
\begin{equation} \label{eq:LZ}
f_{\mathcal{C}^{\prime}}(\eta ):=  \frac{    1+ 2 \eta \mathcal{I}'_4 + \eta^2 }{\sqrt{2} \sqrt{1+\eta^2} \sqrt{2(1-\eta^2)(2\cI_4-1)+(\eta\cI_4'+1)^2}}.
\end{equation}
\end{theorem}

\subsubsection{Doubly static curves}

There are special static curves for which $f_{\mathcal C'}(u)=1$. In order to investigate this case, we need to deeply understand the geometry of static curves. 
\begin{lemma}
\label{lem:sta}
A curve is static if and only if 
$$\int_{0}^{L}\dg_1(t)^2dt=\int_{0}^{L}\dg_2(t)^2dt=L/2 \;\;\; {\rm and} \;\;\; \int_{0}^{L}\dg_1(t)\dg_2(t)dt=0.$$
\end{lemma}
\begin{remark}\rm 
If $\cC$ is static, then also $\int_{0}^{L}\dg_1(t)^4dt=\int_{0}^{L}\dg_2(t)^4dt$. 
\end{remark}
Let us now define
\begin{align*}
A=A_{\cC}&:=\frac{1}{L}\int_{0}^{L}\dg_1(t)^2\dg_2(t)^2dt, \qquad B=B_{\cC}:=\frac{1}{L}\int_{0}^{L}\dg_1(t)^3\dg_2(t)dt,\\
\cI_4&=\cI_{4,\cC}:=\frac{1}{L^2}\int_{0}^{L}\int_{0}^{L}\langle\dg(t),\dg(u)\rangle^4 dt du.
\end{align*}
Hence 
$$
\mathcal I'_4 = 1-8A.
$$
The following result may look technical, but it is instrumental to introduce the notion of a doubly static curve.
\begin{lemma} \label{lem:ABI}
If $\cC$ is a static curve, then we have the relation,
\begin{equation}
\label{eq:I4}
\cI_4=\frac{1}{2}+8A^2-2A+8B^2, 
\end{equation}
and the inequalities
\begin{align*}
0<A<\frac{1}{4}, \hspace{1cm} B^2<\frac{A(1-4 A) }{4}, \hspace{1cm}  \frac{3}{8}\leq\cI_4<\frac{1}{2}, \hspace{1cm}  -1<\cI_4'<1.
\end{align*}
\end{lemma}

The proof of Lemma \ref{lem:ABI} is postponed to Appendix \ref{AuxLem}. We can now introduce the notion of a doubly static curve.

\begin{definition}
We say the curve $\mathcal C$ doubly static if $\cI_4=3/8$.     
\end{definition}

\begin{example} Circles and semicircles are doubly static. 
\end{example}
Our characterization of doubly static curves is given in the following lemmas, also proved in Appendix \ref{AuxLem}.
\begin{lemma}
\label{lem:dou}
One has $\cI_4=3/8$ if and only if $\cC$ is static, $A=1/8$, and $B=0$; this implies also $\cI_4'=0$.
\end{lemma}
\begin{corollary}
Let $\mathcal{C}''\subset \mathbb T^2$ be a doubly static curve. For $\delta$-separated sequences of energy levels $n \subset \{S\}$ such that $\mathcal N_n\to +\infty$
and $\widehat \mu_n(4)\to \eta$, 
\begin{equation}
    \Corr(\mathcal L_n, \mathcal Z_n(\mathcal C'')) \to 1.
\end{equation}
\end{corollary}

\begin{lemma}[cf. {\cite[Appendix G]{roswig}}] \label{lem:roswig}
Let $\cC\subset\Tb^2$ be a smooth closed curve with nowhere $0$ curvature, invariant with respect to rotations by $2\pi/k$, for integer $k=3$ or $k\geq 5$. Then $\cC$ is doubly static.
\end{lemma}


\subsubsection{Discussion}

Our first main results on the asymptotic correlation structure among the functionals that we introduced in Section \ref{secGeometry} can be conveniently summarized in the following (symmetric) correlation matrix. As mentioned above, $\mathcal C'$ (resp. $\mathcal C''$) denotes a static (resp. doubly static) curve. 

\begin{center}
{\rm Asymptotic correlation structure, $d=2$.}
\end{center}
\begin{equation*}
\begin{matrix}
& \mathcal{L}_{n}^{0} & \mathcal{L}_{n}^{u_{1}} & \mathcal{L}_{n}^{u_{2}} & \mathcal{Z}_{n}\mathcal{(C)} & \mathcal{Z}_{n}\mathcal{(C}^{\prime }) & \mathcal{Z}_{n}\mathcal{(C}^{\prime \prime } )  \\ 
\mathcal{L}_{n}^{0} & 1 &  &  &  &  &  \\ 
\mathcal{L}_{n}^{u_{1}} & 0 & 1 &  &  &  &  \\ 
\mathcal{L}_{n}^{u_{2}} & 0 & 1 & 1 &  &  &  \\ 
\mathcal{Z}_{n}\mathcal{(C)} & 0 & 0 & 0 & 1 &  &  \\ 
\mathcal{Z}_{n}\mathcal{(C}^{\prime }) & f_{\mathcal{C}^{\prime}}(\eta ) & 0 & 0 & 0 & 1 &  \\ 
\mathcal{Z}_{n}\mathcal{(C}^{\prime \prime }) & 1 & 0 & 0 & 0 & f_{\mathcal{C}^{\prime}}(\eta ) & 1
\end{matrix}
\end{equation*}

\begin{remark}
\textup{Level curves have asymptotically full correlation at different non-zero thresholds $u_1,u_2$; this is the phenomenon noted by \cite{Wigman2012} for random spherical harmonics, using the expansion of the 2-point correlation function, and then related to the domination of the second chaos by \cite{MPRW, CM18, NourdinPeccatiRossi2019, MRossiWigman2020, MarinucciRossi21, CMR23} and others. On the other hand, similarly to what was noted earlier in \cite{MarinucciRossi21} for eigenfunctions on the sphere, Theorem \ref{thu1u2} shows that the nodal length and the boundary lengths of excursion sets at non-zero levels are asymptotically fully uncorrelated in the high-energy regime. This can be interpreted as a spurious effect: boundary lengths at non-zero levels are dominated by the second-order chaos, which is proportional to the random norm of the eigenfunctions, and the latter of course has no impact on the nodal length (which is invariant to normalizations).  }
\end{remark}

\begin{remark}\rm It is important to note that the correlation between boundary length and the number of nodal intersections with static or non-static curves is always zero in the asymptotic limit, excluding the nodal case (i.e., $u=0$). However, the mechanism here is  different than what we observed in the previous remark: indeed, the second-order chaos component in the intersection of the nodal length with a non-static curve does \emph{not} vanish, although it is still uncorrelated with the random norm of the eigenfunctions, which dominates the behaviour of the boundary length. See the proof of Theorem \ref{19:07} for more details.
On the other hand for intersections with static curves the second order chaos is of lower order, (see \cite{roswig}, Section 2.1) and therefore the asymptotic correlation with the boundary lengths is obviously zero. In some sense, intersections with static curves have some form of invariance with respect to normalization factors for the Arithmetic Random Waves, and this makes their behaviour somewhat analogous to nodal lines, see our following discussion on partial autocorrelation results.     
\end{remark}

\begin{remark}\rm 
 In the special case where $\eta=0$ the limiting spectral measure is Lebesgue, i.e. lattice points are equidistributed in the limit. In these circumstances, we have 
\[
 \mathrm{Corr}(\Lc_n^0,\Zc_n(\mathcal{C}^{\prime} ))  \to \frac{1}{\sqrt{2}} \frac{1}{\sqrt{4\cI_4-1}} ;
\]
this result can be compared to Theorem \ref{theo:3D} below for the three-dimensional case. At the other extreme we have $\eta=\pm 1$, where the limiting measure is Cilleruelo or tilted Cilleruelo, namely the spectral measure exhibits the maximal concentration. In these cases 
$\mathrm{Corr}( \Lc_n^0,\Zc_n( \mathcal{C}^{\prime} ) )  \to 1$ for any static curve $\cC'$.   
\end{remark}
\begin{remark}\rm 
 We have that $\mathrm{Corr}(\Lc_n^0,\Zc_n(\mathcal{C}^{\prime} ))  >0$ for any $\eta$ and any static curve $\cC'$.
\end{remark}
\begin{remark}\rm 
 In the doubly static case, the dominant terms in $\Lc^0[4]$ and $\Zc(\mathcal{C}^{''})[4]$ coincide, up to a factor depending on the energy and $L$, so that we get full correlation for any $\eta\in[-1,1]$. More explicitly,  $\mathcal{L}_{n}^{0}$ and $\mathcal{Z}_{n}(\mathcal{C}^{\prime \prime })$ are asymptotically the same random variable up to a constant. Once again, we recall that circles and semicircles are doubly static.
\end{remark}

\begin{remark}\rm 
 Let $\cC'$ be a fixed toral curve, static, but not doubly static.  Then (as noted above) we have $f_{\cC'}(\eta)=1$ for $\eta=\pm 1$, and $f_{\cC'}(\eta)=0$ has exactly one solution for $\eta\in(-1,1)$.   
\end{remark}

\begin{remark}
\textup{ Above we consider the case $\eta=\pm 1$ under the assumption of well separated sequences of eigenvalues. We stress that it is possible that $\eta=\pm 1$ may not be attainable under the well separated assumption.} 
\end{remark}

\subsection{The partial correlation structure in the $2$-dimensional case}

To get deeper insights into the correlation structure for the geometry of level sets of Arthmetic Random Waves, it is of greater interest to get rid of the effect of the random $L^2$ norm of the eigenfunctions. More precisely, it is of interest to investigate the so-called partial correlation structure, where the effect of the fluctuations in the eigenfunctions norm is removed. 

Let $X,Y,Z$ be square integrable random variables; we define the partial correlation coefficient between $X$ and $Y$ conditional on $Z$ as 
\begin{equation*}
    \Corr_Z (X, Y) := \Corr(X^*, Y^*),
\end{equation*}
where $X$ and $Y$ are the residuals after projecting $X, Y$ onto the explanatory variable $Z$.

In analogous circumstances, it was shown in \cite{MarinucciRossi21} that for random spherical harmonics perfect autocorrelation holds in the high energy limit (see also \cite{CM22} for a similar result on critical points). For Arithmetic Random Waves, the correlation structure is much more subtle, as detailed below.
    \begin{proposition}\label{prop1}
Let $u_1, u_2\in \mathbb R$, then for subsequences of energy levels $\lbrace n\rbrace\subset S$ such that $\widehat{\mu_{n}}(4)\to \eta\in [-1,1]$ as ${\cal N}_n\to +\infty$ we have 
\begin{eqnarray*}
\Cov (\mathcal L_{n}^{u_1}[4], \mathcal L_{n}^{u_2}[4]) \sim \phi(u_1)\phi(u_2) \frac{\pi}{4} \frac{E_n}{\mathcal N_n^2}\left (2a(u_1)a(u_2) - \frac14 (a(u_1) + a(u_2)) + \frac{3+\eta^2}{8^2}\right ).
\end{eqnarray*}
\end{proposition}
Note that for $u_1=u_2=0$, since $a(0)=\frac18$, we retrieve 
\begin{equation*}
    \Var(\mathcal L_n[4]) \sim \frac{4 \pi^2 n}{\mathcal N_n^2}\frac{1+\eta^2}{512},
\end{equation*}
as expected.  
\begin{theorem}\label{prop1u}
Let $u\in \mathbb R$, for a $\delta$-separated subsequences of energies $\{n\}\subset S$ such that ${\cal N}_{n}\to\infty$ and $\widehat{\mu}_{n}(4)\to \eta\in [-1,1]$, we have
\begin{align*}
 \Cov(\mathcal L_{n}^{u}[4], \mathcal Z_n(\mathcal{C}^{\prime})[4])  &\sim \phi(u)\sqrt{\frac{\pi}{2}} \frac{\sqrt{E_n/2}}{\mathcal N_n}   \frac{\sqrt{2n}}{4 \mathcal N_n} L    \frac 1{16} \Big[ 1+ 2 \eta \mathcal{I}'_4 + \eta^2 \Big].
\end{align*}
\end{theorem}

\subsubsection{Discussion}
For 2-dimensional Arithmetic Random Waves, the following asymptotic partial correlation structure holds
\begin{center}
{\rm Partial Correlation structure, dimension $d=2$.} 
\end{center}
\begin{equation*}
\begin{matrix}
& \mathcal{L}_{n}^{0} & \mathcal{L}_{n}^{u_{1}} & \mathcal{L}_{n}^{u_{2}} & \mathcal{Z}_{n}\mathcal{(C)} & \mathcal{Z}_{n}\mathcal{(C}^{\prime }) & \mathcal{Z}_{n}\mathcal{(C}^{\prime \prime }) \\ 
\mathcal{L}_{n}^{0} & 1 &  &  &  &  &  \\ 
\mathcal{L}_{n}^{u_{1}} & M(0,u_{1};\eta ) & 1 &  &  &  &  \\ 
\mathcal{L}_{n}^{u_{2}} & M(0,u_{2};\eta ) & M(u_{1},u_{2};\eta ) & 1 &  &  &  \\ 
\mathcal{Z}_{n}\mathcal{(C)} & 0 & 0 & 0 & 1 &  &  \\ 
\mathcal{Z}_{n}\mathcal{(C}^{\prime }) & f_{\mathcal C'}(\eta) & f_{\mathcal C'}(u_1;\eta) & f_{\mathcal C'}(u_2;\eta) & 0 & 1 &  \\ 
\mathcal{Z}_{n}\mathcal{(C}^{\prime \prime }) & 1 & M(0,u_{1};\eta ) & M(0,u_{2};\eta ) & 0 & f_{\mathcal C'}(\eta) & 1 
\end{matrix}
\end{equation*}
where for $u\in \mathbb R$
\begin{align*}
f_{\mathcal C'}(u;\eta) := \frac{       \frac {\sqrt 2}{16} \Big[ 1+ 2 \eta \mathcal{I}'_4 + \eta^2 \Big]}{\sqrt{ (2a(u)^2-\frac12 a(u)+\frac{3+\eta^2}{8^2})} \cdot     \sqrt{2(1-\eta^2)(2\cI_4-1)+(\eta\cI_4'+1)^2}},
\end{align*} 
(note that $f_{\mathcal C'}(\eta)=f_{\mathcal C'}(0;\eta)$   as in \eqref{eq:LZ} since $a(0) = \frac18$). Moreover,  
\begin{equation*}
M(u_{1},u_{2};\eta )=\frac{\left\{ 2a(u_{1})a(u_{2})-\frac{1}{4}(a(u_{1})+a(u_{2}))+\frac{3+\eta
^{2}}{8^2}\right\} }{\sqrt{\left\{ 2a^{2}(u_{1})-\frac{1}{2}a(u_{1})+\frac{ 3+\eta ^{2}}{8^2}\right\} \left\{ 2a^{2}(u_{2})-\frac{1}{2}a(u_{2})+\frac{ 3+\eta ^{2}}{8^2}\right\} }}, 
\end{equation*}
for
\begin{equation*}
a(u)=\frac{H_{4}(u)}{4}+\frac{H_{2}(u)}{2}-\frac{1}{8}\text{ .}
\end{equation*}
Note that 
\begin{equation*}
    M(0,u;\eta ) = f_{\mathcal C''}(u;\eta),
    \end{equation*}
    where $\mathcal C''$ is a doubly static curve. 

\begin{remark}
\textup{The rationale behind the previous correlation table can be explained as follows. Considering partial correlation, the second-order chaos term disappears from the boundary lengths measure at non-zero thresholds. As a consequence their correlation with the intersections for non static curves becomes zero, because the latter is dominated by the second-order chaos. On the other hand, for static or doubly static curves the second-order chaos is lower order, hence the partial correlation becomes basically the correlation between the fourth-chaos components of intersections and boundary lengths. 
}
\end{remark}

An important consequence of the previous results is the existence of {\it resonant pairs}, that is, sets of threshold levels with asymptotically full correlation between boundary length and/or nodal intersections; this is illustrated in the following corollary.

\begin{corollary}\label{Cor:Correlation Curves} We have that 
\begin{align*}
\lim_{n \rightarrow \infty}{\rm Corr}(\mathcal{L}_{n}^{u_{1}},\mathcal{L}_{n}^{u_{2}})=1, 
\end{align*}
if and only if
\begin{equation}\label{fullcorrelationcurve}
u_2^{4}-4u_2^{2}=u_1^{4}-4u_1^{2}. 
\end{equation}
\end{corollary}

\begin{example}
For $u_1=0$ we obtain that the nodal lines (and the interesections with doubly static curves) have asymptotically perfect correlations with the levels $u_2=\pm 2$. For $u_1=1$ we obtain $u_2^{4}-4u_2^{2}+3=0,$ with solutions $u_2=%
\sqrt{3},-\sqrt{3},-1,1,$ so that resonant pairs are given by $(1,3),(1,-3)$ and $(1,-1)$.    
\end{example}

We call \eqref{fullcorrelationcurve} the \emph{Full Correlation Curve} for boundary lengths of Arithmetic Random Waves. We believe that analogous algebraic curves characterize full correlation pairs for other geometric functionals, such as Lipschitz-Killing curvatures and critical values. We leave the investigation of this for future research.

\subsection{Some results in the 3-dimensional case}\label{3D-results}

It is of obvious interest to investigate the existence of a similar correlation structure for higher-dimensional arithmetic random waves (as studied for instance in \cite{Cam19, benmaf, RWY16}). For brevity's sake, we do not explore fully this possibility here; we focus just on a special case, that is in 3 dimensions the correlation between the nodal area ($\mathcal{A}_{n}=\mathcal{A}_{n}^{0}$) and the length of intersections of the nodal area with static surfaces or doubly-static surfaces ($\mathcal{M}_{n}(\Sigma ^{\prime })$, $\mathcal{M}_{n}(\Sigma ^{\prime \prime }))$. Let us first define 
$$
\mathcal A = \mathcal A_n := \mathcal H^2 (\lbrace x\in \mathbb T^3 : T_n(x)=0\rbrace), 
$$
and, for $\Sigma\subset \mathbb T^3$  a fixed compact regular toral surface, of finite area $A:=|\Sigma|$,
$$
\mathcal M = \mathcal M_n(\Sigma) :=  \mathcal H^1 (\lbrace x\in \mathbb T^3 : T_n(x)=0, x\in \Sigma\rbrace).
$$ 
From \cite[Theorem 1.2]{benmaf}, as $n\to +\infty$, $n\not\equiv 0,4,7 (\mathrm{mod}\; 8)$,
\begin{equation*}
    \Var(\mathcal A_n) \sim \frac{32}{375}\frac{n}{\mathcal N_n^2}.
\end{equation*}
Assume that $\Sigma$ admits a smooth normal vector locally, and call $n(\sigma)$ the unit normal vector to $\Sigma$ at the point $\sigma$. For $k\geq 0$ even, call
\begin{equation}
\cI_k=\cI_{k,\Sigma}:=\frac{1}{A^2}\iint_{\Sigma^2} \langle n(\sigma), n(\sigma')\rangle^k\,d\sigma d\sigma'.
\end{equation}

\begin{definition}
We call $\Sigma$ of nowhere $0$ Gauss-Kronecker curvature {\rm static} if $\cI_2=1/3$, and {\rm doubly static} if $\cI_4=1/5$.    
\end{definition}

\begin{remark}
\textup{ Also for surfaces doubly static implies static (see Lemma \ref{lem:dou3d} below). Simple examples of doubly static surfaces are the sphere and hemisphere.}
\end{remark}
From \cite{mafros}, we have as $n\to +\infty$, $n\not\equiv 0,4,7 (\mathrm{mod}\; 8)$, along a well separated\footnote{See Definition 1.6 in \cite{mafros}.} sequence of eigenvalues
\begin{equation}
    \Var(\mathcal M_n) \sim \frac{\pi^2}{9600} \frac{n}{N^2} (81 \mathcal I_4+35 A^2).
\end{equation}

\begin{remark}
\label{rem:sta}
\textup { Any surface $\Sigma$ of finite area and nowhere $0$ Gauss-Kronecker curvature, invariant with respect to any permutation and sign change of the coordinates is static. To see this, note that under this condition $\Sigma$ verifies the criterion for staticity given by Lemma \ref{lem:sta3d} below. For instance, $\Sigma$ may be given piecewise by symmetric trivariate polynomials where all variables appear to even powers (as long as the assumption on the curvature is met everywhere).}
\end{remark}

Our main result is the following. 

\begin{theorem}
\label{theo:3D}
Let $\mathcal A$ be the nodal area and $\mathcal M$ the nodal intersection length. For static surfaces $\Sigma$ of area $A$, we have as $n\to\infty$, $n \not\equiv 0,4,7 \pmod 8$ along a well separated sequence of eigenvalues
\begin{equation}
\label{cov3d}
\Cov(\mathcal{A}_n,\mathcal M_n)\sim\frac{n}{\Nc^2}\cdot\frac{8\pi A}{375}, 
\end{equation}
so that
\begin{equation*}
\Corr(\mathcal A_n,\mathcal M_n) \longrightarrow \frac{16}{\sqrt{405\cdot\cI_4+175}}. 
\end{equation*}
\end{theorem}

\begin{remark}\rm 
Static surfaces verify $1/5\leq\cI_4\leq 1/3$ (Lemma \ref{lem:dou3d}). The above limit is $1$ for `doubly static' surfaces i.e. those satisfying $\cI_4=1/5$, for instance sphere and hemisphere.
\end{remark}



\subsection{Plan of the paper}
The results concerning the 2-dimensional correlation structure of Arithmetic Random Waves are included in Section \ref{2D1}, whereas those for partial correlation are given in Section \ref{2D2}; Section \ref{3D} is devoted to the arguments related to nodal surfaces, whereas Appendix \ref{AuxLem} collects the proofs for the technical lemmas that we exploited to characterize static and doubly static curves.
\subsection{Acknowledgments}

We are grateful to GNAMPA-INdAM Project 2022 \emph{Proprietà e teoremi limite per funzionali di campi Gaussiani}, to MUR Department of Excellence Project 2023-2027 \emph{MatModTov}, to MUR Prin 2023-2025 \emph{Grafia}, and to Swiss National Science Foundation project 200021\_184927, held by Prof. Maryna Viazovska, for financial support.

\section{The 2-dimensional Correlation Structure}\label{2D1}

Let us start investigating the correlation between the boundary length at level $u$ and the number of nodal intersections with respect to a non-static curve. \\

\begin{proof}[Proof of Theorem \ref{19:07}]
For $u=0$ the result follows immediately from the orthogonality of the projections in the chaos expansion. In fact in \cite[Section 2.1]{roswig} it is shown that, in the case of a non-static curve, the second chaotic projection dominates the chaos expansion of $\mathcal{Z}_n(\mathcal{C})$, while it is known, see \cite[Section 1.4]{MPRW}, that $\mathcal L_n^0$ is dominated by the fourth chaotic projection. For $u \ne 0$ the boundary length is dominated by the second order chaos \cite[Theorem 2.4]{CMR23}, so we have  
    \begin{align*}
        &\Cov(\mathcal L_n^u, \mathcal Z_n(\mathcal C) ) \\
        &= \sqrt{2\pi^2n}\sqrt{\frac{\pi}{8}}\phi(u) u^2 \frac{\sqrt n}{\sqrt 2} L \frac{1}{{\mathcal N}_n^2/4}    \sum_{\lambda, \lambda'\in \Lambda_n^+} \mathbb E[(|a_\lambda|^2 - 1) (|a_{\lambda'}|^2 -1)]( 2 I_{\lambda',\lambda'}(2,0) -1 )  + o\left ( \frac{n}{\mathcal N_n}\right ) \\
        &=  \sqrt{2\pi^2n}\sqrt{\frac{\pi}{8}}\phi(u) u^2   \frac{\sqrt n}{\sqrt 2} L \frac{1}{{\mathcal N}_n^2/4} \sum_{\lambda \in \Lambda_n^+} \mathbb ( 2 I_{\lambda,\lambda}(2,0)-1) + o\left ( \frac{n}{\mathcal N_n}\right ) 
        \end{align*}
where the last step follows by observing that $2|a_\lambda|^2$ has a chi-squared distribution with $2$ degrees of freedom, and by recalling that $a_\lambda$ and $a_{\lambda}'$ are independent for $\lambda \ne \lambda'$. The statement immediately follows by observing that 
  \begin{align*}
  \sum_{\lambda \in \Lambda_n^+} \mathbb (-1 + 2 I_{\lambda,\lambda}(2,0))&=  \frac{1}{2} \sum_{\lambda \in \Lambda_n} \mathbb (-1 + 2 I_{\lambda, \lambda}(2,0)) = 0, 
  \end{align*}
    since, as shown in Lemma \ref{sumeq20} equation \eqref{sumeq20}, we have    
    $$
    \frac{1}{\mathcal N_n} \sum_{\lambda\in \Lambda_n} I_{\lambda, \lambda}(2,0) = \frac 1 2. 
    $$
\end{proof}

The main result to prove  Proposition \ref{we201223}, i.e.   the correlation structure among nodal length and nodal intersection numbers, is the following proposition. 

\begin{proposition}
\label{prop:LZ}
Let $\Lc_n$ be the nodal length, and $\Zc_n(\mathcal{C}')$ the nodal intersections number with a static curve. For a $\delta$-separated subsequence of energies $\{n\}\subset S$ such that ${\cal N}_{n}\to\infty$ and $\widehat{\mu}_{n}(4)\to \eta\in [-1,1]$,  we have
\begin{equation}
\label{eq:covLZ}
 \Cov(\mathcal L_n, \mathcal Z_n(\mathcal{C}^{\prime}) )  \sim  \frac{\sqrt{\pi^2 n}}{\mathcal N_n}   \frac{\sqrt{n}}{4 \mathcal N_n} L    \frac 1{16} \Big[ 1+ 2 \eta \mathcal{I}'_4 + \eta^2 \Big], 
\end{equation}
and
\begin{equation}
\label{eq:LZ}
\Corr(\Lc_n,\Zc_n(\mathcal{C}^{\prime}))\sim   \frac{ 1+ 2 \eta \mathcal{I}'_4 + \eta^2 }{\sqrt 2 \sqrt{1+\eta^2} \sqrt{2(1-\eta^2)(2\cI_4-1)+(\eta\cI_4'+1)^2}}.
\end{equation}
\end{proposition}
\begin{proof}[Proof of Proposition \ref{prop:LZ}]
The covariance in \eqref{eq:covLZ} follows immediately from Proposition \ref{prop1u} with $u=0$.  Now we take into account the expression for $\Var(\Lc_n[4])$ in \cite[(2.20)]{MPRW}, and, in our notation, \cite[(3.4)]{roswig} is
\begin{equation*}
\Var(\Zc[4])\sim\frac{n}{4\mathcal{N}^{2}}\cdot\frac{L^2}{4}[2(1-\eta^2)(2\cI_4-1)+(\eta\cI_4'+1)^2].
\end{equation*}
\end{proof}

\section{The Partial Correlation Structure in the $2$-dimensional Case}\label{2D2}

\begin{proof}[Proof of Proposition \ref{prop1}]
Let us write 
\begin{align*}
    \Cov(\mathcal L_n^{u_1}[4],\mathcal L_n^{u_2}[4] )  \sim &\, \phi(u_1) \phi(u_2) \frac{\pi}{4} \frac{4 \pi^2 n}{\mathcal N_n^2} \\
    &\times \Big \{ a(u_1)a(u_2) \Var(W_1(n)^2) \\ 
    &+ [a(u_1)+a(u_2)]\Cov\Big (W_1(n)^2, - \frac14 W_2(n)^2 - \frac14 W_3(n)^2 - \frac12 W_4(n)^2\Big ) \\
    &+ \Var \Big (- \frac14 W_2(n)^2 - \frac14 W_3(n)^2 - \frac12 W_4(n)^2 \Big)\Big \}.
\end{align*}
From Lemma 4.3 in \cite{MPRW} and a simple computations of Gaussian moments we have, as $n\to +\infty$,
\begin{equation*}
    \Var(W_1(n)^2)\to 2,
\end{equation*}
and moreover 
\begin{equation*}
    \Cov\Big (W_1(n)^2, - \frac14 W_2(n)^2 - \frac14 W_3(n)^2 - \frac12 W_4(n)^2\Big )\to -\frac14.
\end{equation*}
Finally,
\begin{align*}
    &\Var\Big (- \frac14 W_2(n)^2 - \frac14 W_3(n)^2 - \frac12 W_4(n)^2\Big )\\ 
    &\to \frac{1}{16}\cdot 2 \left ( \frac{3+\eta}{8}\right )^2 + \frac{1}{16}\cdot 2 \left ( \frac{3+\eta}{8}\right )^2 + \frac{1}{4}\cdot 2 \cdot  \left ( \frac{1-\eta}{8}\right )^2+ 2\cdot \frac{1}{16}\cdot 2  \left (\frac{1-\eta}{8} \right )^2\\ 
    &= \frac{1}{ 8^2}(3+ \eta^2),
\end{align*}
thus concluding the proof. 

\end{proof}

Let us now investigate partial correlation between the boundary length and the number of intersections with a static curve. Define
\begin{align*}
{\cal I}&:=  \frac{1}{L}\int_0^L\Big [ \frac{3 + \eta}{8} \dg_1^4(t)  + \frac{3 + \eta}{8}  \dg_2^4(t) + 6 \frac{1- \eta}{8}  \dg_1^2(t) \dg_2^2(t)\Big ] dt,\\
{\cal J} &:= \frac{1}{L}\int_0^L \Big[ \Big (\frac{3+\eta}{8}  \Big )^2\dg_1^4(t) + \Big ( \frac{1-\eta}{8}\Big )^2 \dg_1^4(t) + 4\Big ( \frac{3+\eta}{8}\Big ) \Big ( \frac{1-\eta}{8}\Big )\dg_1^2(t) \dg_2^2(t) \\
&\hspace{1.6cm} + \Big ( \frac{1-\eta}{8}\Big )^2 \dg_2^4(t) + \Big (\frac{3+\eta}{8}  \Big )^2\dg_2^4(t) + 8 \Big ( \frac{1-\eta}{8}\Big )^2 \dg_1^2(t) \dg_2^2(t) \Big] dt.
\end{align*}

\begin{proof}[Proof of Theorem \ref{prop1u}]
Using the expressions of the $4$-th order chaos, see Section \ref{sec_chaos}, we write 
\begin{align*}
& \Cov(\mathcal L_{n}^{u}[4], {\mathcal Z}_n(\mathcal C ')[4] ) \\ &\sim \phi(u)\sqrt{\frac{\pi}{2}} \frac{\sqrt{E_n/2}}{\mathcal N_n}   \frac{\sqrt{2n}}{4 \mathcal N_n} L   \;  \mathbb{E} \Big\{ \Big [ a(u) W_1(n)^2 - \frac14 W_2(n)^2 - \frac14 W_3(n)^2 - \frac12 W_4(n)^2  \Big ] \\
 & \hspace{5.4cm} \times \Big [ \frac{2}{ \mathcal{N}_n} \sum_{\lambda, \lambda' \in \Lambda_n} (|a_\lambda|^2-1) (|a_{\lambda'}|^2-1) (1-4 I_{\lambda, \lambda'}(2,2))\Big] \Big\}. 
\end{align*}
The first term is 
\begin{align*}
A&= a(u)  \mathbb{E}\left[  W_1(n)^2 \frac{2}{ \mathcal{N}_n} \sum_{\lambda, \lambda' \in \Lambda^+_n} (|a_\lambda|^2-1) (|a_{\lambda'}|^2-1) (1-4 I_{\lambda, \lambda'}(2,2))  \right] \\
&= a(u)\frac{4}{ {\mathcal N}^2_n} \mathbb{E} \Big [  \Big( \sum_{\lambda_1, \lambda_2 \in \Lambda^+_n} (|a_{\lambda_1}|^2-1) (|a_{\lambda_2}|^2-1)  \Big) \Big( \sum_{\lambda_3, \lambda_4 \in \Lambda^+_n} (|a_{\lambda_3}|^2-1) (|a_{\lambda_4}|^2-1) (1-4 I_{\lambda_3, \lambda_4}(2,2)) \Big) \Big] \\
&= a(u)\frac{4}{ {\mathcal N}^2_n}  \Big [\sum_{\lambda, \lambda' \in \Lambda^+_n} (1-4 I_{\lambda', \lambda'}(2,2))  + 2 \sum_{\lambda, \lambda' \in \Lambda^+_n}(1-4 I_{\lambda, \lambda'}(2,2))  \Big]\\
&=  a(u)\frac{4}{ {\mathcal N}^2_n} \sum_{\lambda, \lambda' \in \Lambda^+_n} (1-4 I_{\lambda', \lambda'}(2,2))  
\end{align*}
where in the last step we apply \eqref{sumdif}.  The second and third term  have the form 
\begin{align*}
B_i&=-\frac 1 4 \frac{4}{n^2 {\mathcal N}^2_n} \Big [\sum_{\lambda, \lambda' \in \Lambda^+_n} \lambda_i^4 (1-4 I_{\lambda', \lambda'}(2,2))  + 2 \sum_{\lambda, \lambda' \in \Lambda^+_n} \lambda_i^2 (\lambda'_i)^2 (1-4 I_{\lambda, \lambda'}(2,2))  \Big]
\end{align*}
for $i=1,2$, and the last term is given by 
\begin{align*}
C&=-\frac 1 2 \frac{4}{n^2 {\mathcal N}^2_n} \Big [ \sum_{\lambda, \lambda' \in \Lambda^+_n} \lambda_1^2 \lambda_2^2 (1-4 I_{\lambda', \lambda'}(2,2))  + 2 \sum_{\lambda, \lambda' \in \Lambda^+_n} \lambda_1 \lambda_2 \lambda'_1 \lambda'_2  (1-4 I_{\lambda, \lambda'}(2,2))  \Big].
\end{align*}
We observe that 
\begin{align*}
B_1+B_2+C&= -\frac 1 4 \frac{4}{n^2 {\mathcal N}^2_n} \Big[ \sum_{\lambda, \lambda' \in \Lambda^+_n} \langle \lambda, \lambda \rangle^2  (1-4 I_{\lambda', \lambda'}(2,2)) + 2\sum_{\lambda, \lambda' \in \Lambda^+_n} \langle \lambda, \lambda' \rangle^2  (1-4 I_{\lambda, \lambda'}(2,2))  \Big].
\end{align*}
In view of Lemma \ref{lemI} we write 
\begin{align*}
A&=a(u) \frac{1}{{\mathcal N}_n^2} [{\mathcal N}_n^2 - 4 {\mathcal N}_n^2 {\mathcal{I}}]= a (u) [1-4 \, \mathcal{I}]\\
B_1+B_2+C&=- \frac 1 4 \frac{1}{n^2 {\mathcal N}_n^2} [n^2 {\mathcal N}_n^2 -4 n^2 {\mathcal N}_n^2 \mathcal{I} + 2\frac 1 2 n^2 {\mathcal N}_n^2- 2 \cdot 4 n^2 {\mathcal N}_n^2 {\mathcal{J}}]= - 1  +4 \mathcal{I} - 2\frac 1 2+ 2 \cdot 4 \mathcal{J};
\end{align*}
i.e. 
\begin{align*}
 A+B_1+B_2+C=    a (u) [1-4 \, \mathcal{I}] -\frac  1 4  + \mathcal{I} - \frac 1 4 + 2 \mathcal{J}  =  \Big [a(u) -\frac 1 4 \Big ] [1-4 \, \mathcal{I}]  - \frac 1 4  + 2 \mathcal{J}.
\end{align*}
Moreover, 
\begin{align*}
- \left( a(u) - \frac14 \right) \frac{1}{\mathcal N_n}\sum_{\lambda\in \Lambda_n} (4\mathcal I_{\lambda, \lambda}(4,0) - 1) = - \left( a(u)-\frac14 \right) (4\mathcal I -1),  
\end{align*}
and 
\begin{align*}
    - \left( a(u) - \frac14 \right) \frac{1}{\mathcal N_n/2}\sum_{\lambda,\lambda'\in \Lambda^+_n} (|a_\lambda|^2-1)(|a_{\lambda'}|^2-1)(1-4\mathcal I_{\lambda,\lambda'}(2,2)) =- \left(a(u) - \frac14 \right) (1-4\mathcal I). 
    \end{align*}
    Finally, 
    \begin{align*}
&\frac{1}{\mathcal N_n}\sum_{\lambda\in \Lambda_n} (4\mathcal I_{\lambda, \lambda}(4,0) - 1)\frac{1}{\mathcal N_n/2}\sum_{\lambda,\lambda'\in \Lambda^+_n} (|a_\lambda|^2-1)\left(|a_{\lambda'}|^2-1)(a(u) - \frac14 \left\langle \frac{\lambda}{|\lambda|}, \frac{\lambda'}{|\lambda'|} \right\rangle^2 \right)\\
&= (4\mathcal I -1) \left (a(u)-\frac14 \right). 
    \end{align*}
So we obtain that 
\begin{align*}
 \Cov(\mathcal L_{n}^{u}[4], \mathcal Z_n(\mathcal{C}^{\prime})[4])  &\sim \phi(u)\sqrt{\frac{\pi}{2}} \frac{\sqrt{E_n/2}}{\mathcal N_n}   \frac{\sqrt{2n}}{4 \mathcal N_n} L    \Big\{  \Big [a(u) -\frac 1 4 \Big ] [1-4 \, \mathcal{I}] - \frac 1 4 + 2 \mathcal{J} - (a(u) -\frac14) (1-4\mathcal I) \Big\}\\
 &= \phi(u)\sqrt{\frac{\pi}{2}} \frac{\sqrt{E_n/2}}{\mathcal N_n}   \frac{\sqrt{2n}}{4 \mathcal N_n} L    \Big\{   - \frac 1 4 + 2 \mathcal{J}  \Big\}.
 \end{align*}
Observing that, 
\begin{align*}
{\cal J} = 5 \frac{2}{8^2} + \frac{4}{8^2} \eta {\cal I}'_4 + \frac{2}{8^2} \eta^2,
\end{align*}
we arrive at
\begin{align*}
 \Cov(\mathcal L_{n}^{u}[4], \mathcal Z(\mathcal{C}^{\prime})[4] )  &\sim   \phi(u)\sqrt{\frac{\pi}{2}} \frac{\sqrt{E_n/2}}{\mathcal N_n}   \frac{\sqrt{2n}}{4 \mathcal N_n} L    \frac 1{16} \Big[ 1+ 2 \eta \mathcal{I}'_4 + \eta^2 \Big]. 
\end{align*}

\end{proof}

\begin{proof}[Proof of Corollary \ref{Cor:Correlation Curves}]
We need to study the covariance expression for the boundary lengths at thresholds $u_1,u_2$, that is
\[
\mathrm {Cov}(\mathcal{L}_{n}^{u_{1}},\mathcal{L}_{n}^{u_{2}})=2a(u_{1})a(u_{2})-\frac{1}{4}(a(u_{1})+a(u_{2}))+\frac{3+\eta^2 }{64}
\]
where we have that
\begin{align*}
a(u) =\frac{u^{4}-6u^{2}+3}{4}+\frac{u^{2}-1}{2}-\frac{1}{8} =\frac{1}{4}u^{4}-u^{2}+\frac{1}{8}. 
\end{align*}
We have that
\begin{align*}
\mathrm {Cov}(\mathcal{L}_{n}^{u},\mathcal{L}_{n}^{0})&=2a(u)a(0)-\frac{1}{4}(a(u)+a(0))+\frac{3+ \eta^2}{64} \\
&=2\left(\frac{1}{4}u^{4}-u^{2}+\frac{1}{8}\right)\frac{1}{8}-\frac{1}{4}\left(\frac{1}{4}
u^{4}-u^{2}+\frac{1}{8}+\frac{1}{8}\right)+\frac{3+ \eta^2}{64} \\
&=\frac{1}{64}(2(2u^{4}-8u^{2}+1)-2(2u^{4}-8u^{2}+2)+3+ \eta^2)\\
&=\frac{1+ \eta^2}{64}.
\end{align*}
On the other hand
\begin{align*}
\mathrm{Var}(\mathcal{L}_{n}^{u})&=\mathrm{Cov}(\mathcal{L}_{n}^{u},\mathcal{L}_{n}^{u})\\
&=2\left(\frac{1}{4}u^{4}-u^{2}+\frac{1}{8}\right)^{2}-\frac{1}{2} \left(\frac{1}{4}u^{4}-u^{2}+\frac{1}{8}\right)+\frac{3+ \eta^2}{64}\\
&=\frac{1}{8}u^{8}-u^{6}+2u^{4}+\frac{1+ \eta^2}{64}=\frac{u^{4}}{8}(u^{2}-4)^{2}+\frac{1+ \eta^2}{64},
\end{align*}
so that the squared correlation is given by
\begin{align*}
\frac{\{\mathrm{Cov}(\mathcal{L}_{n}^{u},0)\}^2}{\mathrm{Var}(\mathcal{L}_{n}^{0}) \mathrm{Var}(\mathcal{L}_{n}^{u})}=\frac{1+ \eta^2}{8u^{4}(u^{2}-4)^{2}+1+ \eta^2},
\end{align*}
and there are resonance points for the nodal length at $u=\pm 2$, because at those points
obviously $\frac{u^{4}}{8}(u^{2}-4)^{2}=0.$ More generally, considering any two threshold levels $u_1,u_2$ we obtain that
\begin{align*}
\mathrm{Cov}(\mathcal{L}_{n}^{u_1},\mathcal{L}_{n}^{u_2})&=2a(u_1)a(u_2)-\frac{1}{4}(a(u_1)+a(u_2))+\frac{3+ \eta^2}{64}\\
&=2\left(\frac{1}{4}u_1^{4}-u_1^{2}+\frac{1}{8}\right) \left(\frac{1}{4}u_2^{4}-u_2^{2}+\frac{1}{8}\right)-
\frac{1}{4}\left(\frac{1}{4}u_1^{4}-u_1^{2}+\frac{1}{8}+\frac{1}{4}u_2^{4}-u_2^{2}+\frac{1}{8}\right)+\frac{3}{64}\\
&= \frac{1}{8}u_1^{4}u_2^{4}-\frac{1}{2}u_1^{4}u_2^{2}-\frac{1}{2}u_1^{2}u_2^{4}+2u_1^{2}u_2^{2}+\frac{1+ \eta^2}{64},
\end{align*}
so that the correlation is one if and only if
$$1=\frac{(\frac{1}{8}u_1^{4}u_2^{4}-\frac{1}{2}u_1^{4}u_2^{2}-\frac{1}{2}
u_1^{2}u_2^{4}+2u_1^{2}u_2^{2}+\frac{1+ \eta^2}{64})^{2}}{(\frac{1}{8}u_1^{8}-u_1^{6}+2u_1^{4}+
\frac{1+ \eta^2}{64})(\frac{1}{8}u_2^{8}-u_2^{6}+2u_2^{4}+\frac{1+ \eta^2}{64})},
$$
and hence, 
\begin{align*}
&\left(\frac{1}{8}u_1^{4}u_2^{4}-\frac{1}{2}u_1^{4}u_2^{2}-\frac{1}{2}
u_1^{2}u_2^{4}+2u_1^{2}u_2^{2}+\frac{1+ \eta^2}{64}\right)^{2}\\
&\;\;-\left(\frac{1}{8}u_1^{8}-u_1^{6}+2u_1^{4}+
\frac{1+ \eta^2}{64}\right) \left(\frac{1}{8}u_2^{8}-u_2^{6}+2u_2^{4}+\frac{1+ \eta^2}{64}\right) \\
&=-\frac{1+ \eta^2}{512}\left( -u_1^{4}+4u_1^{2}+u_2^{4}-4u_2^{2}\right) ^{2}=0. 
\end{align*}
\end{proof}

\section{The 3-dimensional Correlation between Nodal Area and Nodal Intersections}\label{3D}

\begin{proof}[Proof of Theorem \ref{theo:3D}]
We may write \cite[(8.99) and Lemma 8.1]{mafros} as
\begin{eqnarray*}
	\mathcal M[4]\sim\sqrt{\frac{4\pi^2m}{3}} \frac{3\cdot 2\cdot A}{16\cdot 8\cdot \Nc}\left [\frac{32}{15}+\frac{1}{\Nc/2}\sum_{\lambda,\lambda'\in\Lambda_{n_j}^{+}}(|a_\lambda|^2-1)(|a_{\lambda'}|^2-1) \left(3-9I(2,2)+14I(2,0)\right.\right.\\\left.\left.-6\left\langle\frac{\lambda}{|\lambda|},\frac{\lambda'}{|\lambda'|}\right\rangle^2+12\left\langle\frac{\lambda}{|\lambda|},\frac{\lambda'}{|\lambda'|}\right\rangle I(1,1) \right)\right],
	\end{eqnarray*}
with
\[I(k,k')=I_{\lambda,\lambda'}(k,k'):=\frac{1}{A}\int_{\Sigma}\left\langle\frac{\lambda}{|\lambda|},n(\sigma)\right\rangle^k\left\langle\frac{\lambda'}{|\lambda'|},n(\sigma)\right\rangle^{k'}d\sigma.\]
For static surfaces we have (Lemma \ref{lem:sta3d})
\[I(2,0)=\frac{1}{3}, \qquad\qquad I(1,1)=\frac{1}{3}\langle\lambda,\lambda'\rangle\]
hence
\begin{eqnarray}\label{L4}
	\mathcal M[4]\sim\sqrt{\frac{4\pi^2m}{3}} \frac{3\cdot 2}{16\cdot 8\cdot \Nc}\frac{A}{15}\left [32+\frac{5}{\Nc/2}\sum_{\lambda,\lambda'\in\Lambda_{n_j}^{+}}(|a_\lambda|^2-1)(|a_{\lambda'}|^2-1) \left(5-27I(2,2)\right.\right.\notag\\\left.\left.-6\left\langle\frac{\lambda}{|\lambda|},\frac{\lambda'}{|\lambda'|}\right\rangle^2\right)\right].
	\end{eqnarray}
Starting with the case of doubly static $\Sigma$, then
\[I(2,2)=\frac{1}{15}(1+2\langle\lambda,\lambda'\rangle^2)\]
so that
\begin{equation*}
	\mathcal M[4]\sim\sqrt{\frac{4\pi^2m}{3}} \frac{3\cdot 2\cdot A}{16\cdot 8\cdot \Nc}\cdot\frac{16}{15}\left [2+\frac{1}{\Nc/2}\sum_{\lambda,\lambda'\in\Lambda_{n_j}^{+}}(|a_\lambda|^2-1)(|a_{\lambda'}|^2-1) \left(1-3\left\langle\frac{\lambda}{|\lambda|},\frac{\lambda'}{|\lambda'|}\right\rangle^2\right)\right].
	\end{equation*}
Comparing with \cite{Cam19}
\begin{equation}\label{A4}
	\mathcal{A}[4]\sim\frac{\sqrt{m}}{5\sqrt{3}\Nc}\cdot 2\left [2+\frac{1}{\Nc/2}\sum_{\lambda,\lambda'\in\Lambda_{n_j}^{+}}(|a_\lambda|^2-1)(|a_{\lambda'}|^2-1) \left(1-3\left\langle\frac{\lambda}{|\lambda|},\frac{\lambda'}{|\lambda'|}\right\rangle^2\right)\right],
	\end{equation}
we see that in this case the two expressions are the same up to a multiplicative factor depending on the area $A$, and in particular $\text{Corr}(\mathcal{A},\mathcal M)\to 1$.

In the general case, from \eqref{A4} and \eqref{L4} we compute $\Cov(\mathcal{A}[4],\mathcal M[4])$, where many terms cancel out, leaving
\begin{eqnarray}
\label{cov3dpre}
\Cov(\mathcal{A}[4],\mathcal M[4])\sim\frac{m}{\Nc^2}\frac{1}{5\sqrt{3}}\frac{2\pi}{\sqrt{3}} \frac{3\cdot 2}{16\cdot 8}\frac{A}{15}\cdot \frac{10}{\Nc^2}\sum_{\lambda,\lambda'\in\Lambda}\left(5-21\left\langle\frac{\lambda}{|\lambda|},\frac{\lambda'}{|\lambda'|}\right\rangle^2+18\left\langle\frac{\lambda}{|\lambda|},\frac{\lambda'}{|\lambda'|}\right\rangle^4\right.\notag\\\left.-27I(2,2)+81I(2,2)\left\langle\frac{\lambda}{|\lambda|},\frac{\lambda'}{|\lambda'|}\right\rangle^2\right).
\end{eqnarray}
We exchange the order of sums and integral, and apply Lemma \ref{lem:I3d} to simplify \eqref{cov3dpre} to \eqref{cov3d}.

Lastly, we take into account the expressions for $\Var(\mathcal{A})$ \cite[Theorem 1.2]{benmaf} and $\Var(\mathcal M)$ \cite[(1.19)]{mafros} to conclude the proof.
\end{proof}

\subsection{Static surfaces}
\begin{lemma}
\label{lem:I3d}
For any static surface $\Sigma$, we have as $m\to\infty$, $m\not\equiv 0,4,7 \pmod 8$
\begin{align*}
\frac{1}{\Nc^2}\sum_{\lambda,\lambda'\in\Lambda}\left\langle\frac{\lambda}{|\lambda|},\frac{\lambda'}{|\lambda'|}\right\rangle^4\to&\frac{1}{5},
\\
\frac{1}{\Nc^2}\sum_{\lambda,\lambda'\in\Lambda}I(2,2)\to&\frac{1}{9},
\\
\frac{1}{\Nc^2}\sum_{\lambda,\lambda'\in\Lambda}I(2,2)\left\langle\frac{\lambda}{|\lambda|},\frac{\lambda'}{|\lambda'|}\right\rangle^2\to&\frac{11}{15^2}.
\end{align*}
\end{lemma}
\begin{proof}
In each of the three expressions,
we expand the summands and apply \cite[Lemma 2.5]{benmaf}. Then we simply recall that the normal $n$ is of norm one to complete the calculations in the second and third expressions.
\end{proof}

\begin{lemma}
	\label{lem:sta3d}
A surface $\Sigma$ is static if and only if for every $i,j$ one has
\[\int_{\Sigma}n_in_j\, d\sigma=\frac{A}{3}\delta_{ij}.\]	
\end{lemma}	
\begin{proof}
We write
\begin{align}
\label{eq:isum}
\cI:=\iint_{\Sigma^2} \langle n(\sigma), n(\sigma')\rangle^2\,d\sigma d\sigma'=\sum_{i,j}\left(\int_{\Sigma}n_i(\sigma)n_j(\sigma)d\sigma\right)^2\geq\sum_{i}\left(\int_{\Sigma}n_i(\sigma)^2d\sigma\right)^2.
\end{align}
The sum of the three integrals
\[\sum_{i}\int_{\Sigma}n_i(\sigma)^2=A\]
is fixed, so the sum of their squares is smallest when they are all equal:
\begin{align*}
\cI\geq\sum_{i}\left(\frac{A}{3}\right)^2=\frac{A^2}{3}.
\end{align*}
All summands in \eqref{eq:isum} are non-negative: then $\cI$ is minimised, i.e. $\Sigma$ is static, if and only if for every $i,j$ one has
\[\int_{\Sigma}n_in_j\, d\sigma=\frac{A}{3}\delta_{ij}.\]
\end{proof}

\begin{remark}
Surfaces in Remark \ref{rem:sta} that satisfy the further condition $\int n_1^4d\sigma =|\Sigma|/5$ are doubly static, due to Lemma \ref{lem:dou3d}.
\end{remark}
\begin{lemma}
	\label{lem:dou3d}
One has $\cI_4=1/5$ if and only if $\Sigma$ is static and
\begin{equation}
\label{eq:515}
\int n_i n_j n_\ell n_k\, d\sigma = \begin{cases}
A/5\qquad \text{if } i=j=\ell=k,\\
A/15\qquad \text{if } i,j,\ell,k \text{ are pairwise equal},\\
0\qquad \text{otherwise.}
\end{cases}
\end{equation}	
Generic surfaces $\Sigma$ satisfy
$1/5\leq\cI_4\leq 1$
(and the maximum is attained by surfaces contained in a plane). If $\Sigma$ is static, then
\[\frac{1}{5}\leq\cI_4\leq\frac{1}{3}.\]
\end{lemma}	
\begin{proof}
The upper bounds are due to $\cI_4\leq\cI$. For generic $\Sigma$,
\begin{align}
\label{eq:i4sum}
A^2\cI_4&\notag:=\iint_{\Sigma^2} \langle n(\sigma), n(\sigma')\rangle^4\,d\sigma d\sigma'=\sum_{i,j,k,l}\left(\int_{\Sigma}n_in_jn_kn_ld\sigma\right)^2\\&\notag
\geq 3\left[\sum_{i}\frac{1}{3}\left(\int_{\Sigma}n_i^4d\sigma\right)^2+\sum_{i<j}2\left(\int_{\Sigma}n_i^2n_j^2d\sigma\right)^2\right]\\&\notag=3\left[\left(\frac{a_{ii}^2}{9}+\frac{a_{ii}^2}{9}+\frac{a_{ii}^2}{9}+\frac{a_{jj}^2}{9}+\frac{a_{jj}^2}{9}+\frac{a_{jj}^2}{9}+a_{ij}^2+a_{ij}^2+a_{ik}^2+a_{jk}^2\right)\right.\\&\left.+\left(\frac{a_{kk}^2}{9}+\frac{a_{kk}^2}{9}+\frac{a_{kk}^2}{9}+a_{ik}^2+a_{jk}^2\right)\right]
\end{align}
with the notation $a_{ij}:=\int_{\Sigma}n_i^2n_j^2d\sigma$, where $\{i,j,k\}$ is any permutation of $\{1,2,3\}$. On the RHS of \eqref{eq:i4sum} there are two brackets: the former is the sum of squares of ten terms, the latter of another five. The latter five sum up to
\[\frac{a_{kk}}{3}+\frac{a_{kk}}{3}+\frac{a_{kk}}{3}+a_{ik}+a_{jk}=\int_{\Sigma}n_k^2\,d\sigma.\]
Since the sum of all fifteen is simply $A$, the former ten have a total of $A-\int_{\Sigma}n_k^2\,d\sigma$. With the same idea as in Lemma \ref{lem:sta3d},
\begin{equation*}
\cI_4\geq \frac{3}{10}\left[3X^2-2XA+A^2\right]
\end{equation*}
where
\begin{equation*}
X=X_{\Sigma}:=\max_{k=1,2,3}\left(\int_{\Sigma}n_k^2\,d\sigma\right).
\end{equation*}
Now $\sum_{k=1,2,3}\int_{\Sigma}n_k^2\,d\sigma=A$ and the three summands are non-negative, so that $X\geq A/3$. Moreover, if $A$ is fixed, then $3x^2-2xA+A^2$ is an increasing function of $x$ for $x>A/3$. It follows that $\cI_4\geq 1/5$.
\\
In addition, if this minimum is achieved, then necessarily
$\int_{\Sigma}n_k^2\,d\sigma=A/3$ for $k=1,2,3$, and \eqref{eq:515} must hold. It then also follows that 
\begin{equation*}
\int_{\Sigma}n_in_j\,d\sigma=\int_{\Sigma}n_i^3n_j\,d\sigma+\int_{\Sigma}n_in_j^3\,d\sigma+\int_{\Sigma}n_in_jn_k^2\,d\sigma=0, \qquad i\neq j.
\end{equation*}
\end{proof}

\begin{lemma}
	One has $\cI_k=A^2/(k+1)$ if and only if $\Sigma$ is static and
	\begin{equation}
	\label{eq:515gen}
	\int n_1^xn_2^yn_3^zd\sigma= \begin{cases}
	A\frac{(x-1)!!(y-1)!!(z-1)!!}{(k+1)!!}\qquad \text{if } x,y,z \text{ are all even,}\\
	0\qquad \text{otherwise,}
	\end{cases}
	\end{equation}
	for all $x,y,z\geq 0$ satisfying $x+y+z=k$.
	\\
	Generic surfaces $\Sigma$ satisfy
	$A^2/(k+1)\leq\cI_k\leq A^2$
	(and the maximum is attained by surfaces contained in a plane). If $\Sigma$ is static, then
	\[\frac{A^2}{k+1}\leq\cI_k\leq\frac{A^2}{3}.\]
\end{lemma}	
\begin{proof}
	The upper bounds are due to $\cI_k\leq\cI$. For generic $\Sigma$,
	\begin{align}
	\cI_k&\notag:=\iint_{\Sigma^2} \langle n(\sigma), n(\sigma')\rangle^k\,d\sigma d\sigma'=\sum_{x+y+z=k}c_{x,y,z}\left(\int_{\Sigma}n_1^xn_2^yn_3^zd\sigma\right)^2\\&\notag
	\geq
	\sum_{\substack{x+y+z=k\\x,y,z \text{ even}}}c_{x,y,z}\left(\int_{\Sigma}n_1^xn_2^yn_3^zd\sigma\right)^2
	\\&\notag=(k-1)!!\left[
	\sum_{\substack{x+y=k\\x,y\\ \text{ even}}}\frac{c_{x,y,0}}{(k-1)!!}\left(\int_{\Sigma}n_1^xn_2^yd\sigma\right)^2+\sum_{\substack{x+y+z=k\\x,y,z \text{ even}\\z\geq 2}}\frac{c_{x,y,z}}{(k-1)!!}\left(\int_{\Sigma}n_1^xn_2^yn_3^zd\sigma\right)^2\right],
	\end{align}
	with
	\[c_{x,y,z}:=\binom{k}{x\ y\ z }\]
	the trinomial coefficient. In the RHS, we replace each summand with
	\[c'_{x,y,z}:=\binom{k}{x\ y\ z }\frac{(x-1)!!^2(y-1)!!^2(z-1)!!^2}{(k-1)!!}\]
	copies of
	\[\left(\int_{\Sigma}\frac{n_1^xn_2^yn_3^z}{(x-1)!!(y-1)!!(z-1)!!}d\sigma\right)^2.\]
	The sum of all integrals is $A$, and there are a total of
	\[\sum_{\substack{x+y+z=k\\x,y,z \text{ even}}}c'_{x,y,z}=(k+1)!!\]
	integrals. The second sum contains a third of the terms, and they sum up to 
	$\int_{\Sigma}n_3^2\,d\sigma$. With the same idea as in Lemma \ref{lem:sta3d},
	\begin{equation*}
	\cI_k\geq(k-1)!! \left[\frac{(A-X)^2}{2(k+1)!!/3}+\frac{X^2}{(k+1)!!/3}\right]=\frac{3}{2(k+1)}\left[3X^2-2XA+A^2\right]
	\end{equation*}
	where
	\begin{equation*}
	X=X_{\Sigma}:=\max_{k=1,2,3}\left(\int_{\Sigma}n_k^2\,d\sigma\right).
	\end{equation*}
	Now $\sum_{k=1,2,3}\int_{\Sigma}n_k^2\,d\sigma=A$ and the three summands are non-negative, so that $X\geq A/3$. Moreover, if $A$ is fixed, then $3x^2-2xA+A^2$ is an increasing function of $x$ for $x>A/3$. It follows that $\cI_k\geq A^2/(k+1)$.
	\\
	In addition, if this minimum is achieved, then necessarily
	$\int_{\Sigma}n_k^2\,d\sigma=A/3$ for $k=1,2,3$, and \eqref{eq:515gen} must hold. It then also follows that 
	\begin{equation*}
	\int_{\Sigma}n_in_j\,d\sigma=\int_{\Sigma}n_in_j(n_1^2+n_2^2+n_3^2)^{(k-2)}\,d\sigma=0, \qquad i\neq j.
	\end{equation*}
\end{proof}

\appendix 
\section{Proofs of auxiliary lemmas}\label{AuxLem}
\begin{proof} [Proof of Lemma \ref{lem:sta}]
Let $\cC$ be a curve parameterized by arc-lenght, and $\nu$ be a probability measure on $\cS^1$. We can assume that $\nu$ is invariant w.r.t. rotations by $\pi/2$, indeed the curve is static if and only if $4\mathcal B_{\mathcal C}(\nu) - L^2=0$ holds for $\nu$ the uniform measure. Via some manipulations, we rewrite
\begin{equation}\label{relaz1}\mathcal{B}_{\cC}(\nu)=\frac{L^2}{8}(1+2\cI_2)+\frac{L^2}{8} \widehat{\nu}(4) (1-2\cI_2^{\perp});
\end{equation}
with $\widehat{\nu}(k)=\int_{\cS^1} z^{-k} d \nu(z)$, for any $k \in \mathbb{Z}$, Fourier transform of $\nu$; note that $\widehat{\nu}(4)$ is real if $\nu$ is invariant under the transformations $z \to \bar{z}$ and $z \to i \cdot z$; and 
\begin{align*}
\cI_2&=\cI_{2,\cC}=\frac{1}{L^2}\int_{0}^{L}\int_{0}^{L}\langle\dg(t),\dg(u)\rangle^2 dt du, \\
 \cI_2^{\perp}&=\cI_{2,\cC}^{\perp}=\frac{1}{L^2}\int_{0}^{L}\int_{0}^{L}\langle\dg(t),(\dg_2(u),\dg_1(u))\rangle^2 dt du.
 \end{align*}
Then $\cC$ is static if and only if $\cI_2=\cI_2^{\perp}=1/2$: indeed if $\cI_2=\cI_2^{\perp}=1/2$ then from \eqref{relaz1} it holds that $4\mathcal B_{\mathcal C}(\nu)-L^2=0$ in particular for $\nu$ the uniform measure. On the other hand, we may rearrange
\[\cI_2=\sum_{i,j=1,2}\left(\frac{1}{L}\int_{0}^{L}\dg_i(t)\dg_j(t)dt\right)^2\geq\sum_{i=1,2}\left(\frac{1}{L}\int_{0}^{L}\dg_i(t)^2dt\right)^2.\]
Since $\int_{0}^{L}\dg_1(t)^2dt+\int_{0}^{L}\dg_2(t)^2dt=L$, then $\cI_2\geq 1/2$ with equality if and only if $\int_{0}^{L}\dg_1(t)^2dt=\int_{0}^{L}\dg_2(t)^2dt=L/2$ and $\int_{0}^{L}\dg_1(t)\dg_2(t)dt=0$. These conditions also ensure that $\cI_2^{\perp}=1/2$, and the proof of this lemma is complete.
\end{proof}\\

\noindent \begin{proof}[Proof of Lemma \ref{lem:dou}]
Similarly to the proof of Lemma \ref{lem:sta}, we have 
\begin{equation}
\label{eq:I4expand}
\cI_4=
\sum_{i=0}^{4}\binom{4}{i}\left[\frac{1}{L}\int_{0}^{L}\dg_1(t)^i\dg_2(t)^{4-i}dt\right]^2
\geq\sum_{i=1,2}\left(\frac{1}{L}\int_{0}^{L}\dg_i(t)^4dt\right)^2+6\left(\frac{1}{L}\int_{0}^{L}\dg_1(t)^2\dg_2(t)^2dt\right)^2.
\end{equation}
Clearly $\int_{0}^{L}\dg_1(t)^4dt+\int_{0}^{L}\dg_2(t)^4dt+2\int_{0}^{L}\dg_1(t)^2\dg_2(t)^2dt=L$, hence $\cI_4\geq 3/8$ 
with equality iff
\begin{equation}
\label{eq:dou}
\frac{1}{L}\int_{0}^{L}\dg_i(t)\dg_j(t)\dg_k(t)\dg_l(t)dt= \begin{cases}
3/8\qquad \text{if } i=j=k=l,\\
1/8\qquad \text{if } i=j\neq k=l,\\
0\qquad \text{\;\;\, otherwise.}
\end{cases}
\end{equation}
If \eqref{eq:dou} holds true, this clearly means that $A=1/8$, and $B=0$, and moreover $\cC$ is static due to Lemma \ref{lem:sta} 
(e.g. $\int_{0}^{L}\dg_1(t)^2dt=\int_{0}^{L}\dg_1(t)^4dt+\int_{0}^{L}\dg_1(t)^2\dg_2(t)^2dt=3L/8+L/8=L/2$). 
\\
Vice versa, assume that $\cC$ is static, $A=1/8$, and $B=0$. By staticity we have $\int_{0}^{L}\dg_1(t)^4dt=\int_{0}^{L}\dg_2(t)^4dt$. Using $A=1/8$ and the fact that $\cC$ is unit speed we get also the first case of \eqref{eq:dou}. As for the third one, it suffices to point out that $\int_{0}^{L}\dg_1(t)\dg_2(t)^3dt=-B=0$. 
\end{proof}\\

\noindent \begin{proof}[Proof of Lemma \ref{lem:roswig}]
To construct a family of doubly static curves, we adapt \cite[Appendix G]{roswig}. The condition is $\cI_{4,\cC}=3/8$. Bearing in mind \cite[(G.4)]{roswig},
we impose
\[\sum_{j=0}^{k-1}\cos\left(\dg(t)-\phi(u)+j\cdot\frac{2\pi}{k}\right)^4=\frac{3k}{8}\]
where $\dg(u)=\exp(i\phi(u))$. Due to the identity
\[\cos(x)^4=\frac{3}{8}+\frac{\cos(2x)}{2}+\frac{\cos(4x)}{8},\]
it suffices to impose
\[\frac{k}{\gcd(2,k)}\geq 2 \qquad\text{and}\qquad \frac{k}{\gcd(4,k)}\geq 2,\]
i.e. $k=3$ or $k\geq 5$. For $k=4$ the curve is static, but not necessarily doubly static (for $\cC$ to be static we only need $k/\gcd(2,k)\geq 2$).
\end{proof} \\

\noindent \begin{proof}[Proof of Lemma \ref{lem:ABI}]
To prove \eqref{eq:I4}, we rewrite $\cI_4$ as
\[\cI_4=\sum_{i=0}^{4}\binom{4}{i}\left[\frac{1}{L}\int_{0}^{L}\dg_1(t)^i\dg_2(t)^{4-i}dt\right]^2.\]
The terms for $i=1,2$ are simply $4B^2$ and $6A^2$. Since $\cC$ is static, in light of Lemma \ref{lem:sta} the term for $i=3$ is $4(-B)^2$, 
and the terms for $i=0,4$ are each equal to $(1/2-A)^2$. Rearranging proves \eqref{eq:I4}.
\\
For static curves $(\int\dg_i^4dt/L)+A=1/2$, with $i=1,2$. Moreover, by Cauchy-Schwartz, $A$ is always the smaller of the two summands 
, hence $0<A<1/4$ (extrema excluded else $\cC$ would be a straight line segment). Rearranging, $-1<\cI_4'<1$.
\\
The inequality $3/8\leq\cI_4$ is shown in Lemma \ref{lem:dou}. For the upper bound, by Cauchy-Schwartz 
$\cI_4<\cI_2$ (as defined in Lemma \ref{lem:sta}), and $\cI_2=1/2$ due to staticity (extremum excluded again else $\cC$ would be a straight line segment). Lastly, we combine $\cI_4<1/2$ and \eqref{eq:I4} to find $B^2<A(1-4A)/4$.
\end{proof}\\

To state the technical results in Lemma \ref{lemI}, we introduce the following notation 
\begin{align*}
{\cal I}&:=  \frac{1}{L}\int_0^L\Big [ \frac{3 + \eta}{8} \dg_1^4(t)  + \frac{3 + \eta}{8}  \dg_2^4(t) + 6 \frac{1- \eta}{8}  \dg_1^2(t) \dg_2^2(t)\Big ] dt,\\
{\cal J} &:= \frac{1}{L}\int_0^L \Big[ \Big (\frac{3+\eta}{8}  \Big )^2\dg_1^4(t) + \Big ( \frac{1-\eta}{8}\Big )^2 \dg_1^4(t) + 4\Big ( \frac{3+\eta}{8}\Big ) \Big ( \frac{1-\eta}{8}\Big )\dg_1^2(t) \dg_2^2(t) \\
&\hspace{1.6cm} + \Big ( \frac{1-\eta}{8}\Big )^2 \dg_2^4(t) + \Big (\frac{3+\eta}{8}  \Big )^2\dg_2^4(t) + 8 \Big ( \frac{1-\eta}{8}\Big )^2 \dg_1^2(t) \dg_2^2(t) \Big] dt
\end{align*}
and we observe that 
\begin{align*}
{\cal I}&= \frac{3}{8} + \frac{\eta}{8} {\cal I}'_4, \hspace{1cm} {\cal J} = 5 \frac{2}{8^2} + \frac{4}{8^2} \eta {\cal I}'_4 + \frac{2}{8^2} \eta^2.
\end{align*}
\begin{lemma}\label{lemI}
 We have 
 \begin{align}
 & \frac{1}{\mathcal N_n} \sum_{\lambda\in \Lambda_n} I_{\lambda, \lambda}(2,0) = \frac 1 2, \label{sumeq20}\\ 
 &   \frac{1}{\mathcal N_n} \sum_{\lambda\in \Lambda_n} I_{\lambda, \lambda}(2,2) = {\mathcal I},  \label{sumeq} \\
 &  \frac{1}{\mathcal N_n^2} \sum_{\lambda, \lambda'\in \Lambda_n} I_{\lambda, \lambda'}(2,2) = \frac1 4.  \label{sumdif} \\
& \frac{1}{n^2 {\mathcal N}_n^2} \sum_{\lambda, \lambda'\in \Lambda_n} \langle \lambda, \lambda'\rangle^2  
= \frac 1 2,  \label{bra}\\ 
&\frac{1}{n^2 {\mathcal N}_n^2} \sum_{\lambda, \lambda'\in \Lambda_n} \langle \lambda, \lambda'\rangle^2 I_{\lambda,\lambda'}(2,2) 
= {\cal J}.  \label{brasumdif}
\end{align}
\end{lemma}
\begin{proof} 
Equation \eqref{sumeq20} follows immediately by observing that 
\begin{align*}
\frac{1}{ \mathcal{N}_n} \sum_{\lambda\in \Lambda_n}  I_{\lambda,\lambda}(2,0) &= 
\frac{1}{n \mathcal{N}_n} \sum_{\lambda \in \Lambda_n}  \frac{1}{L}\int_0^L [ \lambda_1 \dg_1(t) + \lambda_2 \dg_2 (t)]^2  d t \\ 
&= \frac{1}{n \mathcal{N}_n} \sum_{\lambda \in \Lambda_n}  \frac{1}{L}\int_0^L 
[\lambda_1^2  \dg_1^2(t) +  \lambda_2^2 \dg_2^2(t) + 2 \lambda_1 \lambda_2 \dg_1(t) \dg_2(t) ] dt\\
&= \frac{1}{L}\int_0^L\Big [\frac 1 2 \dg_1^2(t)  + \frac 1 2  \dg_2^2(t) \Big ] dt \\
&= \frac 1 2  \frac{1}{L} \int_0^L \langle \dg(t), \dg(t) \rangle  \, dt = \frac 1 2,
\end{align*}
where we used the relations 
\begin{equation}\frac{1}{n \mathcal{N}_n} \sum_{\lambda \in \Lambda_n} \lambda_i^2=\frac 1 2, \;\; i=1,2, \hspace{0.5cm} {\rm and } \hspace{0.5cm} \frac{1}{n \mathcal{N}_n} \sum_{\lambda \in \Lambda_n} \lambda_1 \lambda_2=0. \label{ig} \end{equation}
The proof of \eqref{sumeq} is similar, in fact  
\begin{align*}
\frac{1}{ \mathcal{N}_n} \sum_{\lambda\in \Lambda_n}  I_{\lambda,\lambda}(2,2) &= 
\frac{1}{n^2 \mathcal{N}_n} \sum_{\lambda \in \Lambda_n}  \frac{1}{L}\int_0^L [ \lambda_1 \dg_1(t) + \lambda_2 \dg_2(t)]^4  d t \\ 
&= \frac{1}{n^2 \mathcal{N}_n} \sum_{\lambda \in \Lambda_n}  \frac{1}{L}\int_0^L [ \lambda_1^4  \dg_1^4(t) +  \lambda_2^4 \dg_2^4(t) + 6 \lambda_1^2 \lambda_2^2 \dg_1^2(t) \dg_2^2(t) ] dt\\
&= {\cal I},
\end{align*}
since 
$$\frac{1}{n^2 \mathcal{N}_n} \sum_{\lambda \in \Lambda_n} \lambda_i^4=\frac{3+\eta}{8}, \;\; i=1,2, \hspace{0.5cm} {\rm and } \hspace{0.5cm} \frac{1}{n^2 \mathcal{N}_n} \sum_{\lambda \in \Lambda_n} \lambda_1^2 \lambda_2^2=\frac{1-\eta}{8}.$$
To prove \eqref{sumdif}, we observe that  
\begin{align*}
\frac{1}{ \mathcal{N}_n^2} \sum_{\lambda, \lambda'\in \Lambda_n}  I_{\lambda,\lambda'}(2,2) &= 
\frac{1}{n^2 \mathcal{N}_n^2} \sum_{\lambda, \lambda'\in \Lambda_n}  \frac{1}{L}\int_0^L [\lambda_1 \dg_1(t) + \lambda_2 \dg_2(t)]^2 [\lambda'_1 \dg_1(t) + \lambda'_2 \dg_2(t)]^2 d t \\ 
&= \frac{1}{n^2 \mathcal{N}_n^2} \sum_{\lambda, \lambda'\in \Lambda_n}  \frac{1}{L}\int_0^L [ \lambda_1^2 \dg_1^2(t) + \lambda_2^2 \dg_2^2(t) + 2 \lambda_1\lambda_2\dg_1(t) \dg_2(t)] \\
& \hspace{3.4cm} \times  [ (\lambda'_1)^2 \dg_1^2(t) + (\lambda'_2)^2 \dg_2^2(t) + 2 \lambda'_1\lambda'_2\dg_1(t) \dg_2(t) ] d t \\
&= \frac{1}{n^2 \mathcal{N}_n^2} \sum_{\lambda, \lambda'\in \Lambda_n}  \frac{1}{L}\int_0^L 
[ \lambda_1^2 (\lambda_1')^2 \dg_1^4(t) + \lambda_1^2 (\lambda_2')^2 \dg_1^2(t) \dg_2^2(t)  \\
&\hspace{3.6cm} + \lambda_2^2 (\lambda_1')^2 \dg_1^2(t) \dg_2^2(t) + \lambda_2^2 (\lambda'_2)^2 \dg_2^4(t) ] dt\\
&= \frac{1}{L}\int_0^L\Big [ \frac14 \dg_1^4(t) + \frac12 \dg_1^2(t) \dg_2^2(t) + \frac14 \dg_2^4(t) \Big ] dt \\
&= \frac14 \frac{1}{L}\int_0^L \langle \dg(t), \dg(t) \rangle^2 dt =\frac14.
\end{align*}
Equation \eqref{bra} is an immediate consequence of \eqref{ig}. Equation \eqref{brasumdif} follows from  
\begin{align*}
&\frac{1}{n^2 \mathcal{N}_n^2} \sum_{\lambda, \lambda'\in \Lambda_n} \langle \lambda, \lambda'\rangle^2 I_{\lambda,\lambda'}(2,2) \\
&= 
\frac{1}{n^4 \mathcal{N}_n^2} \sum_{\lambda, \lambda'\in \Lambda_n} (\lambda_1 \lambda'_1+\lambda_2\lambda_2')^2 \frac{1}{L}\int_0^L [\lambda_1 \dg_1(t) + \lambda_2 \dg_2(t)]^2 [\lambda'_1 \dg_1 (t)+ \lambda'_2 \dg_2(t) ]^2 d t \\ 
&= \frac{1}{n^4 \mathcal{N}_n^2} \sum_{\lambda, \lambda'\in \Lambda_n} \{ \lambda^2_1 (\lambda'_1)^2+\lambda_2^2(\lambda_2')^2 + 2 \lambda_1\lambda_2\lambda_1'\lambda_2'\} \\
&\;\; \times \frac{1}{L}\int_0^L [ \lambda_1^2 \dg_1^2(t) + \lambda_2^2 \dg_2^2 (t)+ 2 \lambda_1\lambda_2\dg_1(t) \dg_2(t) ]\, [ (\lambda'_1)^2 \dg_1^2(t) + (\lambda'_2)^2 \dg_2^2(t) + 2 \lambda'_1\lambda'_2\dg_1(t) \dg_2(t) ] d t \\
&= \frac{1}{n^4 \mathcal{N}_n^2} \sum_{\lambda, \lambda'\in \Lambda_n}  \frac{1}{L}\int_0^L [\lambda_1^4 (\lambda_1')^4 \dg_1^4(t) + \lambda_1^2 (\lambda_1')^2 \lambda_2^2 (\lambda_2')^2 \dg_1^4 (t) + \lambda_1^4 (\lambda_1')^2 (\lambda_2')^2 \dg_1^2(t) \dg_2^2(t)  \\
&\hspace{3cm}+ \lambda_1^2 \lambda_2^2 (\lambda_2')^4\dg_1^2(t) \dg_2^2(t)+ \lambda_1^2\lambda_2^2 (\lambda_1')^4\dg_1^2(t) \dg_2^2 (t)+ (\lambda'_2)^2\lambda_2^4 (\lambda_1')^2\dg_1^2(t) \dg_2^2(t)\\
&\hspace{3cm}+ \lambda_1^2 (\lambda_1')^2\lambda_2^2 (\lambda_2')^2 \dg_2^4(t) +  \lambda_2^4 (\lambda_2')^4 \dg_2^4(t)+ 8 \lambda_1^2 \lambda_2^2 (\lambda_1')^2(\lambda_2')^2 \dg_1^2 (t)\dg_2^2 (t)] dt\\ 
&= {\mathcal J}.
\end{align*}

\end{proof}

\bibliographystyle{alpha}
\bibliography{bibfile.bib}
\end{document}